\documentclass[notitlepage,11pt,reqno]{amsart}

\usepackage{amsopn,esint,nicefrac}
\usepackage{mathabx}
\usepackage[final]{hyperref}
\usepackage[T1]{fontenc}
\usepackage{graphicx}

\usepackage[utf8]{inputenc}
\usepackage{apptools}
\AtAppendix{\counterwithin{lemma}{section}} 
\usepackage[margin=1in]{geometry}
\allowdisplaybreaks
\newcommand{\stkout}[1]{\ifmmode\text{\sout{\ensuremath{#1}}}\else\sout{#1}\fi}

\newcommand{\R}{\mathbb{R}}

\newcommand{\Rn}{\mathbb{R}^n}
\usepackage{graphicx,enumitem,dsfont,upgreek}
 \usepackage[dvips]{epsfig}
 \usepackage[mathscr]{eucal}
\usepackage{amscd}
\usepackage{amssymb,nicefrac}
\usepackage{amsthm}
\usepackage{amsmath}
\usepackage{latexsym}
\usepackage{dsfont}
\usepackage{upref}
\usepackage{hyperref}

\usepackage{color}
\theoremstyle{plain}

\newtheorem{thm}{Theorem}[section]
\theoremstyle{plain}
\newtheorem{lem}[thm]{Lemma}

\newtheorem{defi}[thm]{Definition}
\newtheorem{rem}{Remark}[section]
\theoremstyle{definition}
\newtheorem*{maintheorem*}{Main Theorem}
\newtheorem*{maincorollary*}{Main Corollary}
{%
\setcounter{enumi}{0}

\begin{enumerate}}%
{\end{enumerate} }

{%
\setcounter{enumi}{0}

\begin{enumerate}}%
{\end{enumerate} }

\newcommand{\norm}[1]{\ensuremath{\left\|#1\right\|}}

\newcommand{\cone}{\ensuremath{\mathcal{C}}}
\newcommand{\cD}{\ensuremath{\mathcal{D}}}
\newcommand{\sL}{\ensuremath{\mathscr{L}}}

\newcommand{\cI}{\ensuremath{\mathcal{I}}}

\newcommand{\dist}{{\rm dist}}
\newcommand{\xbar}{\ensuremath{\bar{x}}}
\newcommand{\ybar}{\ensuremath{\bar{y}}}
\newcommand{\abar}{\ensuremath{\bar{a}}}
\newcommand{\tail}{\texttt{Tail}}
\newcommand{\grad}{\nabla}

\newcommand{\dz}{\ensuremath{\, {\rm d}z}}

\newcommand{\dr}{\ensuremath{\, {\rm d}r}}

\numberwithin{equation}{section} \allowdisplaybreaks

\usepackage{cancel,pdfsync}

\title[Regularity of fractional $p$-Laplacian with coercive gradients]{Lipschitz regularity for fractional 
$p$-Laplacian with coercive gradients}

\begin{document}

\author{Anup Biswas}

\author{Aniket Sen}

\address{Indian Institute of Science Education and Research-Pune, Dr.\ Homi Bhabha Road, Pashan, Pune 411008, INDIA. Emails:
{\tt anup@iiserpune.ac.in, aniket.sen@students.iiserpune.ac.in}}

\author{Erwin Topp}

\address{
Instituto de Matem\'atica, Universidade Federal do Rio de Janeiro, Rio de Janeiro - RJ, 21941-909, BRAZIL; 
Email: {\tt etopp@im.ufrj.br}}

\begin{abstract}
In this article we study nonlinear nonlocal equations with coercive gradient nonlinearity of the form
\[
(-\Delta_{p})^{s}u(x)+H(x,\grad u)=f ,
\]
where $f$ is Lipschitz continuous and the Hamiltonian $H$ satisfies the coercivity condition
\[
H(x,\xi)\;\geq\;\theta|\xi|^{m}-\theta_{1}
\qquad\text{for all } x,\xi\in\Rn,
\]
for some constants $\theta,\theta_{1}>0$ and $m>1$. Under an additional local Lipschitz assumption on $H$, we show that every viscosity solution $u$ is locally Lipschitz continuous provided that
\[
\text{either}\; sp>p-2\; \text{or}\; 1<p<m+1.
\]
We also establish H\"older and Lipschitz continuity of subsolutions, for which boundedness of $f$ alone suffices. Furthermore, when $f=0$ and $H$ is independent of $x$, we prove that the equation admits only the trivial solution within the class of bounded solutions, for all $m,p\in(1,\infty)$.
\end{abstract}


\maketitle


\section{Introduction}
In this article, we investigate the regularity properties of (sub)solutions to  
\begin{equation}\label{main}
\sL u := (-\Delta_p)^s u + H(x, \nabla u) = f \quad \text{in } B_2,
\end{equation}
where
$$
(-\Delta_p)^s u = {\rm PV} \int_{\mathbb{R}^n} |u(x) - u(x+z)|^{p-2} (u(x) - u(x+z)) \frac{dz}{|z|^{n+sp}},
$$
with $p>1$, $s \in (0,1)$, $n \geq 2$, and $f \in C(\overline{B}_2)$. Here, $B_r$ denotes the ball of radius $r$ centered at the origin.
The coercive Hamiltonian $H \in C(\mathbb{R}^n \times \mathbb{R}^n)$ is assumed to satisfy the following conditions:
\begin{itemize}
\item[\hypertarget{H1}{(H1)}] There exists $m > 1$ such that, for each $R>0$ there exists $C_{H,R}>0$ so that
\begin{equation*}
|H(x, \xi_1+\xi_2)-H(y, \xi_1)|\leq C_{H,R}\left[ |x-y| (1+|\xi_1|^m) + |\xi_2| (|\xi_1|^{m-1}+|\xi_2|^{m-1}+1)\right]
\end{equation*}
for all $ x, y\in B_R$ and $\xi_1, \xi_2\in\Rn$ with $|\xi_2| \leq R$.

\item[\hypertarget{H2}{(H2)}] There exist positive constants
$\theta, \theta_1$ such that
$$\theta|\xi|^m-\theta_1 \leq H(x, \xi) 
\quad \text{for all}\; x, \xi\in\Rn.$$
\end{itemize}
A standard example of Hamiltonian satisfying the above conditions would be
$$H(x, \xi)=\langle \xi, a(x) \xi\rangle^{\frac{m}{2}} + b(x)\cdot \xi\quad x, \xi\in\Rn, m>1,$$
where $a:\Rn\to \R^{n\times n}, b:\Rn\to\Rn$ are Lipschitz continuous and $a$ is uniformly positive definite.

Throughout this work, (sub)solutions are understood in the viscosity sense. 
The precise definition of a viscosity solution will be given in the next section (see Definition~\ref{Def2.1}).

Given the structure of \eqref{main}, the notion of viscosity solution is more appropriate in the present setting. 
Indeed, defining a weak solution would require $u \in W^{s,p}_{\mathrm{loc}}(B_2)$; however, $W^{s,p}$ regularity does not provide sufficient control over the gradient nonlinearity. 
As a result, the notion of weak solution becomes technically delicate unless one imposes \emph{a priori} higher regularity on $u$. 
Although the viscosity framework requires continuity of the solution to begin with, it is well suited to the structure of the problem considered here.

To state our main results, we recall the by now classical notion of \texttt{Tail}.
By $L^{p-1}_{sp}(\Rn)$ we denote the weighted $L^p$ space or \texttt{tail space}, defined by
$$L^{p-1}_{sp}(\Rn)=\left\{f\in L^{p-1}_{\rm loc}(\Rn)\; :\; \int_{\Rn}\frac{|f(z)|^{p-1}}{(1+|z|)^{n+sp}}\dz <\infty\right\}.$$
Associated to this tail space we also define the {\it tail function} given by
$$
\tail_{s,p}(f; x, r) := \left(r^{sp} \int_{|z-x|\geq r}\frac{|f(z)|^{p-1}}{|z-x|^{n+sp}}\dz \right)^{\frac{1}{p-1}},\quad r>0.
$$
Let
\begin{equation}\label{A}
A=\sup_{B_2}|u| + \tail_{s,p}(u; 0, 2).
\end{equation}

Our first main result concerns the regularity of the subsolutions.  For this result we do not require $m>1$.

\begin{thm}\label{Tmain-1}
Let $p \in (1,\infty)$ and $m > sp$. 
Let $u \in L^\infty(\overline{B}_2) \cap L^{p-1}_{sp}(\mathbb{R}^n)$ be a viscosity solution of
\begin{equation}\label{Emain-1}
(-\Delta_p)^{s} u + H(x, \grad u) \le D 
\quad \text{in } B_2,
\end{equation}
where $ D > 0$. Assume that \hyperlink{H2}{(H2)} holds.
Then $u$ is $\gamma$-H\"older continuous in $B_1$, and 
$\|u\|_{C^{0,\gamma}(\bar{B}_1)}$ is bounded by a constant depending on 
$D$, $A$, $\theta$,$\theta_1$, $s$, $p$, $n$, $m$, and $\gamma$, where
\[
\gamma =
\begin{cases}
\dfrac{m-sp}{m-(p-1)} & \text{if } sp > p-1, \\[6pt]
\text{any value in } (0,1) & \text{if either } sp = p-1 \text{ or } sp < m \le p-1, \\[6pt]
1 & \text{if } sp < p-1 < m.
\end{cases}
\]
\end{thm}
\begin{rem}\label{R1.1}
For $n> sp$ and $m>sp>p-1$, the above regularity is optimal. In fact, letting $\varphi(x)=|x|^{\frac{m-sp}{m-p+1}}$, it can be easily checked that
$(-\Delta_p)^s \varphi= \kappa |x|^{(p-1)\frac{m-sp}{m-p+1}-sp}$ for $x\neq 0$ and some $\kappa<0$, see 
\cite[Theorem~1.1]{DPQ25}. Thus, for some suitable
$\theta>0$, $\varphi$ is a subsolution to
$(-\Delta_p)^s\varphi+\theta|\grad\varphi|^m=0$
in $\Rn$.
\end{rem}
The above result is in the spirit of the work of Capuzzo-Dolcetta, Leoni, and Porretta~\cite{CDLP} (see also \cite{CV22}), where the authors establish regularity results for subsolutions of superquadratic second-order elliptic equations. 
Remarkably, it was shown in~\cite{CDLP} that the 
H\"{o}lder seminorm of $u$ does not depend on the $L^\infty$ norm of $u$ or on its oscillation.
In the case of the fractional Laplacian, that is, when $p=2$, a similar result was obtained by Barles \emph{et al.}~\cite{BKLT}. 
However, in that setting, the H\"{o}lder seminorm depends on the $L^\infty$ bound of $u$ (or equivalently, on the oscillation of $u$). 
This dependence is essentially unavoidable due to the presence of the nonlocal integral term.

In our next result we investigate regularity of solutions.
\begin{thm}\label{Tmain-2}
Assume that $p \in (1,\infty), m\in (1, \infty)$ and \hyperlink{H1}{(H1)}-\hyperlink{H2}{(H2)} hold. 
Let $u \in C(\overline{B}_2) \cap L^{p-1}_{sp}(\mathbb{R}^n)$ be a viscosity solution of
\begin{equation}\label{ETmain-2}
(-\Delta_p)^{s} u + H(x, \grad u) = f 
\quad \text{in } B_2,
\end{equation}
where $f \in C^{0,1}(\overline{B}_2)$ and $\theta>0$. Assume further that
$$
\frac{sp+1}{p-1} > 1.
$$
Then $u$ is Lipschitz continuous in $\bar{B}_1$, and
\[
\|u\|_{C^{0,1}(\bar{B}_1)} \le C,
\]
where the constant $C$ depends only on $A$, $s$, $p$, $n$, $H$, and $\|f\|_{C^{0,1}(\bar{B}_2)}$. Moreover, when $sp>p-1$, the same Lipschitz estimate holds for $f\in C(\bar{B}_2)$
and the constant $C$ depends only on $A$, $s$, $p$, $n$, $H$, and $\|f\|_{L^\infty(\bar{B}_2)}$.
\end{thm}

The arguments in the proof of Theorem~\ref{Tmain-2} yield the following regularity result. In this regime, the gradient nonlinearity is of lower order relative to the diffusion; consequently, the regularity is driven by the fractional $p$-Laplacian.
\begin{thm}\label{T1.3}
Assume that $1< m \leq sp\leq p-2$, $f\in C^{0,1}(\bar{B}_2)$ and let $u$ be a viscosity solution to~\eqref{ETmain-2}. 
Also, suppose that the Hamiltonian satisfies \hyperlink{H1}{(H1)}.
Then $u \in C^\gamma_{\rm loc}(B_2)$ for every $\gamma \in (0, \frac{sp - m + 1}{p - m - 1})$ and $\norm{u}_{C^{0, \gamma(B_1)}}\leq C$, where the constant $C$ depends on
$A$, $s$, $p$, $n$, $m$, $\gamma$, $H$, and $\|f\|_{C^{0,1}(B_2)}$.
\end{thm}

Our regularity results are summarized in the table below
(combining Theorems~\ref{Tmain-1}, ~\ref{Tmain-2} and ~\ref{T1.3}). As mentioned before, (a)-(c) and a subcase of (d) (corresponding to $sp>p-1$) hold for merely continuous $f$.

\begin{center}
\begingroup
\renewcommand{\arraystretch}{1.5} 
\begin{tabular}{|c|l|c|c|}
\hline 
& \textbf{Solution type} & \textbf{Conditions on parameters} & \textbf{Regularity} \\ [2pt]
\hline
(a) & Subsolution & $m>sp>p-1$ & $\frac{m-sp}{m-(p-1)}$-H\"older \\[2pt] 
\hline 
(b) & Subsolution & $m>sp$, and either $sp=p-1$  & $\gamma$-H\"older, for any $\gamma \in (0,1)$ \\ [-3pt]
&&or $sp<m \leq p-1$&\\[2pt]
\hline 
(c) & Subsolution & $sp<p-1<m$ & Lipschitz \\ [2pt]
\hline 
(d) & Solution & $sp>p-2 $ & Lipschitz \\ [2pt]
\hline 
(e) & Solution & $1<m \leq sp \leq p-2$ & \begin{minipage}{3cm}$\gamma$-H\"older, for any $\gamma\in (0, \frac{sp-m+1}{p-m-1})$
\end{minipage}\\ [2pt]
\hline 
\end{tabular} 
\endgroup
\end{center}

\medskip

The regularity theory for classical quasilinear equations with an additional gradient nonlinearity goes back to the work of Ladyzhenskaya and Ural'tseva \cite{LU68} and Lieberman \cite{Lib91}. There the gradient nonlinearity is subject to a growth restriction, which allows the problem to be treated by variational methods. If instead one imposes a coercivity structure on the gradient term, equations whose gradient nonlinearity has arbitrarily high growth become tractable.

For the Laplace equation (that is, $s=1$ and $p=2$) with a coercive gradient term, the first gradient upper bound was obtained by Lions \cite{Lions}, by means of the Bernstein technique introduced in \cite{Bern06, Bern10}. For the $p$-Laplacian, a gradient estimate up to the boundary was established by Bidaut-V\'eron, Garc\'{\i}a-Huidobro and V\'eron \cite{BHV14} in the range $p>m>p-1$, where the Bernstein method was combined with a Keller--Osserman type construction of radial supersolutions. We also mention \cite{BHV20, CC15, CC20, DGS09}, where a coercive gradient nonlinearity is treated together with the $p$-Laplacian, and \cite{CD90, Pan} in the context of homogenization. In the viscosity framework, gradient bounds were derived in \cite{Bar91, CDLP}. Coercive gradient terms arise naturally in control theory, making such estimates crucial to the analysis of ergodic control problems (see, e.g., \cite{ABC19,BM16,BF92} and the references therein). Moreover, gradient estimates were used in \cite{BHV14} to detect potential isolated singularities and boundary blow-up of solutions, whereas in \cite{CC20} they play an essential role in establishing a priori bounds for solutions.

For nonlocal operators, by contrast, an analogue of the Bernstein estimate remains an open problem. In the case $p=2$, Lipschitz regularity for bounded and uniformly continuous viscosity solutions was established in \cite{BLT17} (see also \cite{BT16} for the subdiffusive case), the coercivity of the Hamiltonian being a key ingredient. Later, for $2s>1$, the authors of \cite{BT24} obtained a Lipschitz estimate for general viscosity solutions of \eqref{main} by combining an Ishii-Lions type argument \cite{IL90} with the H\"older regularity result of \cite{BKLT}. Although Lipschitz regularity is now available in this setting, the sharpness of the estimate remains unclear.

The regularity theory for the fractional $p$-Laplace equation is itself an active and rapidly evolving area. Early contributions may be found in
\cite{Coz17, DKP14, DKP16, KKP16}, and the first sharp H\"older regularity results were obtained in \cite{BL17, BLS18} for $p\geq 2$ and later in \cite{GL24} for $p\in(1,2)$. Since then there has been a surge of work on the regularity of solutions to $(-\Delta_p)^su=f$; see, for instance, \cite{BDLM25, BDLMS24a, BDLMS24b, DKLN, DN25, IMS20}. Lipschitz regularity was first established in \cite{BT25} and subsequently extended in \cite{BS25} to the fractional $p$-Poisson equation with H\"older continuous source terms. Very recently a major advance was achieved in \cite{GJS25}, where $C^{1,\alpha}$ regularity was proved for fractional $p$-harmonic functions in the range $p\in[2,\tfrac{2}{1-s})$, or equivalently $p\geq 2 $ and $sp+1>p-1$. For Lipschitz regularity in the parabolic setting we refer to \cite{JSU} and \cite{GJS26} for the $C^{1, \alpha}$ estimate. We note that the condition $\frac{sp+1}{p-1}>1$ holds automatically for $p\in(1,2]$ and is equivalent to $sp>p-2$ for $p>2$; thresholds of this type were first identified in \cite{BDLMS24a, BDLMS24b} and, independently, in \cite{DKLN}.

Let us now comment on the interplay between the two terms in \eqref{main}.
Under the scaling $u_{\lambda}(x)=u(\lambda x)$ one has
\[
(-\Delta_{p})^{s}u_{\lambda}(x)=\lambda^{sp}\big[(-\Delta_{p})^{s}u\big](\lambda x),
\qquad
|\nabla u_{\lambda}(x)|^{m}=\lambda^{m}|\nabla u(\lambda x)|^{m},
\]
so the gradient nonlinearity is of \emph{higher} order than the diffusion precisely when $m>sp$. In this regime \eqref{main} is not a perturbation of the fractional $p$-Laplace equation: the first-order term \emph{governs} the local behaviour of solutions, and the H\"older exponent in Theorem~\ref{Tmain-1} is determined by $m$ rather than by the diffusion, this exponent being optimal by Remark~\ref{R1.1}. It is therefore not at all obvious that Lipschitz regularity should hold under the condition $\frac{sp+1}{p-1}>1$, as the dominant gradient term cannot by itself produce it; Theorem~\ref{Tmain-2} shows that the diffusion nevertheless does.

Theorem~\ref{Tmain-2} extends the results of \cite{BKLT, BT24, BHV14} to nonlinear, degenerate fractional operators. Our approach is based on the nonlocal Ishii-Lions method, originally introduced in \cite{BCI11, BCCI12} for fractional Laplacian-type operators and later adapted to the nonlinear setting in \cite{BT25}; see also \cite{CGT22}.

Next, we establish a Liouville-type result for our operator.

\begin{thm}\label{Tmain-3}
Suppose that $p \in (1,\infty)$ and $m \in (1,\infty)$. 
Let $H \in C(\mathbb{R}^n)$ be a Hamiltonian satisfying the following conditions:
\begin{itemize}
\item[(i)] There exists a constant $C>0$ such that
\begin{equation*}
|H(\xi_1+\xi_2)-H(\xi_1)|
\le C\,(|\xi_1|^{m-1}+|\xi_2|^{m-1}+1)|\xi_2|
\quad \text{for all } \xi_1,\xi_2 \in \mathbb{R}^n.
\end{equation*}

\item[(ii)] There exist positive constants $\theta$ and $\theta_1$ such that
\[
\theta |\xi|^m - \theta_1 \le H(\xi)
\quad \text{for all } \xi \in \mathbb{R}^n.
\]
\end{itemize}
Then any bounded viscosity solution of
\begin{equation}\label{Emain-3A}
(-\Delta_p)^s u + H(\nabla u) = 0 
\quad \text{in } \mathbb{R}^n
\end{equation}
must be a constant.
\end{thm}

In the case of the $p$-Laplacian, Liouville results of this type are typically derived from gradient estimates obtained via the Bernstein method, which relies in addition on the homogeneity of $H$; see \cite{BHV14, BHV19}. Another classical approach for nonlinear elliptic operators is the nonlinear capacity method of Mitidieri and Pohozaev \cite{MP99, MP01, F09}, which employs test functions of the form $u^{-d}\chi^{k}$, with $\chi$ a smooth cut-off function, and derives suitable integral estimates.

Neither technique extends readily to nonlocal operators; we refer the reader to the recent survey of Cirant and Goffi \cite{CG23}, which highlights several open problems concerning Liouville properties for nonlocal equations. One may ask whether the local regularity estimate of Theorem~\ref{Tmain-2} can be combined with the intrinsic scaling property of \eqref{Emain-3A} (see, for instance, \cite{BHV14}) to recover the Liouville property. While this appears plausible, it would require an explicit dependence of the constant $C$ in Theorem~\ref{Tmain-2} on the parameter $A$; since the proof of Theorem~\ref{Tmain-2} relies in certain cases on a bootstrapping argument, tracking this dependence is technically involved. Moreover, the lack of homogeneity of the operator prevents us from normalizing the constants in any straightforward manner.

A first significant advance toward such results was achieved in \cite{BQT25}, where an Ishii-Lions type argument was used to establish a Liouville theorem for the fractional Laplacian with gradient nonlinearity. The fractional $p$-Laplacian, however, is nonlinear, which introduces substantial difficulties in adapting that method directly: although the general philosophy behind the proof of Theorem~\ref{Tmain-3} is similar, the present setting is considerably more technical. We also mention the recent work \cite{BBS26}, where the Liouville property is established for non-negative supersolutions under the constraint $m<\frac{n(p-1)}{n-sp+p-1}$.

We close this introduction by comparing our approach with the existing ones and indicating where the main difficulties lie. As noted above, gradient estimates for local operators rest on the Bernstein method \cite{BHV14, Lions}, which has
no known nonlocal counterpart, while the Lipschitz estimate of \cite{BLT17}, although based on the coercivity of the Hamiltonian, requires bounded and uniformly continuous solutions and a \emph{linear} nonlocal operator. We therefore base the proofs of Theorems~\ref{Tmain-2} and \ref{T1.3} on the Ishii-Lions method initiated in \cite{BT25} and refined in \cite{BS25}. In particular, we impose no restriction of the form $p>m>p-1$ as in \cite{BHV14}.

The overall line of argument parallels \cite{BS25,BT25}, where no gradient term is present, but the contribution of the Hamiltonian requires care, especially when $m\geq sp$. The proof proceeds by a bootstrap in which the H\"older exponent is improved by a fixed amount at each step. The difficulty is that for $m\geq sp$ this iteration \emph{cannot be started}: it fails unless the solution is already known to possess a minimal degree of H\"older regularity. This is precisely what Theorem~\ref{Tmain-1} supplies. Thus the two
terms of the equation play complementary and clearly separated roles: the coercive Hamiltonian yields H\"older regularity up to a threshold exponent, and from that threshold onward it is the diffusion that drives the iteration to the Lipschitz estimate.

For Theorem~\ref{Tmain-3} we combine the approach of \cite{BQT25}---carried out there for the fractional Laplacian, where the Ishii-Lions argument is considerably simpler (the calculation of $I_3$ is much straightforward there)---with the ideas of \cite{BT25}. A further difficulty arises here: whereas the estimates of \cite{BT25} are performed on small balls, the proof of Theorem~\ref{Tmain-3} requires the Ishii-Lions argument to be carried out on balls of arbitrarily large radius. The parameters must accordingly be chosen with some care for the two sets of ideas to be compatible.

The remainder of the article is organized as follows. In Section~\ref{s-Tmain-1} we introduce the definition of viscosity solutions and prove Theorem~\ref{Tmain-1}. Section~\ref{s-Tmain-2} is devoted to the proofs of Theorems~\ref{Tmain-2} and ~\ref{T1.3}, and Theorem~\ref{Tmain-3} is established in Section~\ref{s-Tmain-3}.

Throughout the paper, we use the notations $C, C_1, C_2, .., \kappa,\kappa_1, \kappa_2,...$ to denote generic constants whose values may vary from line to line.

\section{Viscosity solution and Proof of Theorem~\ref{Tmain-1}}\label{s-Tmain-1}
The (sub)solutions in this article are understood in the viscosity sense, which we define below in the spirit of  \cite{KKL19}. 
To introduce the definition, we recall some notation from 
\cite{KKL19}.  Since, as established in \cite{KKL19}, the operator $\sL$ may not be classically defined for all $C^2$ functions, we must restrict our consideration to a suitable subclass of test functions when defining viscosity solutions.
Given an open set $D$, we denote by $C^2_\eta(D)$ the subset of $C^2(D)$, defined by
$$
C^2_\eta(D)=\left\{\phi\in C^2(D)\; :\; \sup_{x\in D}\left[\frac{\min\{d_\phi(x), 1\}^{\eta-1}}{|\nabla\phi(x)|} +
\frac{|D^2\phi(x)|}{(d_\phi(x))^{\eta-2}}\right]<\infty\right\},
$$
where
$$ 
d_\phi(x)=\dist(x, N_\phi)\quad \text{and}\quad N_\phi=\{x\in D\; :\; \nabla\phi(x)=0\}.$$

The above restricted class of test functions becomes necessary to define $\sL$ in the classical sense in the singular case, that is, for
$p\leq \frac{2}{2-s}$. 
Now we are ready to define the viscosity solution from \cite[Definition~3]{KKL19}. We denote by
$$\sL u(x) = (-\Delta_p)^s u(x) + H(x, \grad u(x))\quad \text{and}\quad \widehat\sL u(x)= (-\Delta_p)^s u(x) + \theta |\grad u(x)|^m.$$

\begin{defi}\label{Def2.1}
A function $u:\Rn\to \R$ is a viscosity subsolution (supersolution) to $\sL u= f$ in $\Omega$ if it satisfies the following
\begin{itemize}
\item[(i)] $u$ is upper (lower) semicontinuous in $\Omega$.
\item[(ii)] If $\varphi\in C^2(B_r(x_0))$ for some $B_r(x_0)\subset \Omega$ satisfies $\varphi(x_0)=u(x_0)$,
$\varphi\geq u$ ($\varphi\leq u$) in $B_r(x_0)$ and one of the following holds
\begin{itemize}
\item[(a)] $p>\frac{2}{2-s}$ or $\nabla\varphi(x_0)\neq 0$,

\item[(b)] $p\leq \frac{2}{2-s}$, $\nabla\varphi(x_0)= 0$ and $x_0$ is an isolated critical point of $\varphi$, and
$\varphi\in C^2_\eta(B_r(x_0))$ for some $\eta>\frac{sp}{p-1}$,
\end{itemize}
then we have 
$$\sL \varphi_r(x_0)  \leq f(x_0) \quad \left(\sL \varphi_r(x_0) \geq f(x_0)\right),$$
 where
\[
\varphi_r(x)=\left\{\begin{array}{ll}
\varphi(x) & \text{for}\; x\in B_r(x_0),
\\[2mm]
u(x) & \text{otherwise}.
\end{array}
\right.
\]

\item[(iii)] $u_+\in L^{p-1}_{sp}(\Rn)$ ($u_-\in L^{p-1}_{sp}(\Rn)$, respectively).
\end{itemize}
A  viscosity solution of $\sL u= f$ in $\Omega$ is both a subsolution and a supersolution in $\Omega$. 
\end{defi}

Now we will prove Theorem~\ref{Tmain-1} with the help of a comparison principle and Definition~\ref{Def2.1}. First, we construct 
(classical) supersolutions to $\widehat\sL u=D_1$, where $D_1$ will be chosen appropriately. For $\upkappa\in (0, 1], r>0$, we define
$v_{\upkappa,r} : \Rn \rightarrow \R$ as:
\begin{equation*}
v_{\upkappa, r} (x)=
\begin{cases}
|x|^{\upkappa} \quad &\text{if } |x| \leq r,
 \\
r^{\upkappa}  &\text{if } |x| > r.
\end{cases}
\end{equation*}
It is easy to see that $v_\upkappa$ is globally $\upkappa$-H\"older continuous. For $C>0$, let us define
$$\tilde{v}_{\upkappa, r}(x)= C r^{-\upkappa}v_\upkappa(x).$$

\begin{lem}\label{L2.2}
There exist $r\in (0,1], C_0>0$, dependent on $\upkappa, D_1, s, p, m, n$, such that
\begin{itemize}
\item[(i)] for any $C>C_0$ and $m>p-1> sp$, we have $\widehat\sL\tilde{v}_{1,1}> D_1$ for $x\in B_1\setminus\{0\}$,
\item[(ii)] for any $C>0$ and $\upkappa\in (\frac{sp}{p-1}, 1)$ and $m\in (sp, p-1]$, we have $\widehat\sL\tilde{v}_{\upkappa,r}> D_1$ for $x\in B_r\setminus\{0\}$. In this case, $r$ also depends on $C$.
\end{itemize}
\end{lem}

\begin{proof}
We start with (i). Note that $\tilde{v}_{1,1}$ is globally Lipschitz with Lipschitz constant being $C$.
Consider $x \in B_1 \setminus \{0\}$ and compute
\begin{align*}
(-\Delta_p)^s \tilde{v}_{1,1}(x)&= {\rm PV}\int_{\Rn}|\tilde{v}_{1,1}(x)-\tilde{v}_{1,1}(x+z)|^{p-2} (\tilde{v}_{1,1}(x)-\tilde{v}_{1,1}(x+z)) \frac{\dz}{|z|^{n+sp}}
\\
&= {\rm PV}\int_{B_1} |\tilde{v}_{1,1}(x)-\tilde{v}_{1,1}(x+z)|^{p-2} (\tilde{v}_{1,1}(x)-\tilde{v}_{1,1}(x+z)) \frac{\dz}{|z|^{n+sp}}
\\
&\quad + \int_{B_1^c} |\tilde{v}_{1,1}(x)-\tilde{v}_{1,1}(x+z)|^{p-2} (\tilde{v}_{1,1}(x)-\tilde{v}_{1,1}(x+z)) \frac{\dz}{|z|^{n+sp}}
\\
&\geq -C^{p-1} \int_{B_1} |z|^{p-1} \frac{\dz}{|z|^{n+sp}} - C^{p-1} \int_{B_1^c} |v_{1,1}(x+z) -|x||^{p-1} \frac{\dz}{|z|^{n+sp}}
\\
&\geq -C^{p-1} \int_0^1 t^{p-1-sp-1} {\rm d}t \int_{\mathbb S^{n-1}} {\rm d}\theta - C^{p-1} \int_{B_1^c}  \frac{\dz}{|z|^{n+sp}}
\\
&\geq -\kappa C^{p-1}
\end{align*}
for some positive constant $\kappa$, independent of $x$, where in the fifth line we use the fact that $p-1>sp$.
Therefore, for $0<|x|<1$, we get
\begin{align*}
\widehat\sL\tilde{v}_{,1}1(x) &= (-\Delta_p)^s \tilde{v}_{1,1}(x) + \theta |\nabla \tilde{v}_{1,1}(x)|^m 
\\
&\geq -\kappa C^{p-1} + \theta C^m
\\
&= C^{p-1} ( -\kappa +\theta C^{m-(p-1)})
\\
&> D_1
\end{align*}
for all $C>C_0$, provided we choose $C_0$ large enough depending on $\kappa, m, p$ and $\theta$, where we use the fact that $m>p-1$.

Now consider (ii). Let $\upkappa\in (\frac{sp}{p-1}, 1)$ and note that $\tilde{v}_{\upkappa, r}$ is globally $\upkappa$-H\"{o}lder with H\"{o}lder constant being $Cr^{-\upkappa}$.
Let $x \in B_r \setminus \{0\}$ and  compute as before to obtain
\begin{align*}
(-\Delta_p)^s\tilde{v}_{\upkappa,r}(x)&={\rm PV}\int_{\Rn}
|\tilde v_{\upkappa,r}(x)-\tilde v_{\upkappa,r}(x+z)|^{p-2} 
(\tilde v_{\upkappa,r}(x)-\tilde v_{\upkappa,r}(x+z)) \frac{\dz}{|z|^{n+sp}}
\\
&\geq -(Cr^{-\upkappa})^{p-1}\int_{B_{2r}} |z|^{\upkappa(p-1)} \frac{\dz}{|z|^{n+sp}} 
- (Cr^{-\upkappa})^{p-1} \int_{B_{2r}^c} |v_{\upkappa,r}(x+z) - |x|^\upkappa |^{p-1} \frac{\dz}{|z|^{n+sp}}
\\
&= - (Cr^{-\upkappa})^{p-1} \int_0^{2r} t^{\upkappa(p-1)-sp-1} {\rm d}t \int_{\mathbb S^{n-1}} {\rm d}\theta 
-  C^{p-1} \int_{B_{2r}^c}  \frac{\dz}{|z|^{n+sp}}
\\
&\geq -\kappa C^{p-1} r^{-sp} - \kappa C^{p-1} r^{-sp}
\\
&\geq -2\kappa C^{p-1} r^{-sp}
\end{align*}
 for some constant $\kappa$. Therefore, for $x\in B_r\setminus\{0\}$, we have
\begin{align*}
\widehat\sL \tilde{v}_{\upkappa,r} (x) = (-\Delta_p)^s \tilde{v}_{\upkappa,r}(x) + \theta |\grad \tilde{v}_{\upkappa,r}(x)|^m
&\geq - 2\kappa C^{p-1} r^{-sp} + \theta (Cr^{-\upkappa})^m \upkappa^m |x|^{m(\upkappa-1)}
\\
&\geq -2\kappa C^{p-1} r^{-sp} + \theta (Cr^{-\upkappa})^m \upkappa^m r^{m(\upkappa-1)}
\\
&\geq  C^m r^{-m} \left[-2\kappa C^{(p-1)-m}  r^{m-sp}+\theta\upkappa^m \right]
\\
&>D_1,
\end{align*}
provided we choose $r\in (0, 1)$ small enough depending on $C, \theta,\upkappa, m$ and $D_1$. This completes the proof.
\end{proof}
Now suppose $sp\geq p-1$. Then for $\upkappa\in (0, 1)$ and $\phi_\upkappa(x)=|x|^\upkappa$, we obtain from \cite[Theorem~1.1]{DPQ25}
that
$$ |(-\Delta_p)^s\phi_\upkappa(x)|=c(\upkappa) |x|^{\upkappa(p-1)-sp}\quad \text{for}\; x\neq 0.$$
Thus, if $m>sp\geq p-1$, for any $C>0$ we have
\begin{align*}
\widehat\sL (C\phi_\upkappa(x)) \geq -c(\upkappa) C^{p-1} |x|^{\upkappa(p-1)-sp} + \theta C^m \upkappa^m |x|^{m(\upkappa-1)}\quad \text{for}\; x\neq 0.
\end{align*}
Set $\upkappa=\frac{m-sp}{m-(p-1)}$ if $sp>p-1$, and $\upkappa\in (0, 1)$ if $sp=p-1$. Such a choice of $\upkappa$ gives us
$0>\upkappa(p-1)-sp\geq m(\upkappa-1)$. Thus, given $D_1$, we find $C_0$ large enough so that for $r=\frac{1}{2}$ we get
$$\widehat\sL(C\phi_\upkappa(x))> D_1\quad \text{for}\; x\in B_r\setminus\{0\},$$
and for all $C\geq C_0$. 

Now we prove Theorem~\ref{Tmain-1}
\begin{proof}[Proof of Theorem~\ref{Tmain-1}]
First, we modify $u$ to be a globally bounded subsolution. Let $\frac{3}{2}<\varrho_1<\varrho_2<2$.
Let
$\chi:\Rn\to [0, 1]$ be a smooth cut-off function satisfying $\chi=1$ in $B_{\varrho_1}$ and $\chi=0$ on $B^c_{\varrho_2}$. Letting, 
$w=\chi u$, it is easy to see from \eqref{Emain-1} that, for $x\in B_{\frac{3}{2}}$ and $p\geq 2$,
\begin{align*}
&\sL w(x) 
\\
&\leq \sL u(x) + (p-1) 2^{p-2} \int_{\Rn} (|u(x+z)-u(x)|+|w(x+z)-u(x)|)^{p-2} |(1-\chi(x+z)) u(x+z)|\frac{\dz}{|z|^{n+sp}}
\\
&\leq D + (p-1)2^{p-2}\int_{|z|\geq \frac{1}{2}(\varrho_1-\frac{3}{2})} (|u(x+z)-u(x)|+|w(x+z)-u(x)|)^{p-2} |(1-\chi(x+z)) u(x+z)|\frac{\dz}{|z|^{n+sp}}
\\
&\leq D + \kappa \int_{\Rn} (|u(x)|^{p-1}+|u(z)|^{p-1}) \frac{\dz}{1+|z|^{n+sp}}\leq D+ \kappa_1 A^{p-1}
\end{align*}
for some constants $\kappa, \kappa_1$, and $A$ is given by \eqref{A}. A similar calculation is also possible for $p\in (1,2)$. Thus,
for some constant $D_1$, dependent on $D, A$, we have
\begin{equation}\label{ET1.1A}
\sL w\leq D_1-\theta_1 \quad \text{in}\; B_{\frac{3}{2}} \Rightarrow \widehat\sL w\leq D_1 \quad \text{in}\; B_{\frac{3}{2}},
\end{equation}
using (H2) and $w=\chi u\in L^\infty(\Rn)$. Now from Lemma~\ref{L2.2} and the discussion following the lemma, we have $r\in (0, \frac{1}{2}]$ and
a function $\varphi$ satisfying the following.
\begin{itemize}
\item[(a)] $\varphi(x)=C |x|^\gamma$ in $B_r$, where $\gamma$ is given by Theorem~\ref{Tmain-1} and the constant $C$ is chosen large enough
so that $\varphi(x)> 2\sup|w|$ in $B^c_r$.
\item[(b)] We have $\widehat\sL\varphi> D_1 $ in $B_r\setminus\{0\}$.
\end{itemize}
We claim that for any $x, y\in B_1$ we have
\begin{equation}\label{ET1.1B}
 w(x)-w(y)\leq \varphi(x-y).
\end{equation}
It is easily seen that Theorem~\ref{Tmain-1} follows from \eqref{ET1.1B}. To prove \eqref{ET1.1B} we fix $y\in B_1$ and
let 
$$M = \sup \{ w(x) -w(y) -\varphi(x-y) : x \in\Rn\}.$$
By the choice of $\varphi$ we see that $w(x) -w(y) -\varphi(x-y)<0$ for $|x-y|\geq r$. Now if $M\leq 0$, then there is nothing to prove and \eqref{ET1.1B} follows. So assume that
$M>0$. Since $w$ is upper semicontinuous, the supremum is attained at some point $x^*\in B_r(y)$. In other words,
$$w(x)\leq M+w(y) + \varphi(x-y)\quad \text{and}\quad w(x^*)=M+w(y) + \varphi(x^*-y).$$
Again, since $M>0$, $x^*\neq y$, and therefore, $\grad(M+w(y) + \varphi(x^*-y))\neq 0$. For $\delta<\frac{1}{2} |x^*-y|$, we define
\[
\varphi_\delta(x)=\left\{\begin{array}{ll}
M+w(y) + \varphi(x-y) & \text{for}\; x\in B_\delta(x^*),
\\[2mm]
w(x) & \text{otherwise}.
\end{array}
\right.
\]
Applying Definition~\ref{Def2.1} in \eqref{ET1.1A} we then have $\widehat\sL\varphi_\delta(x^*)\leq D_1$. Again, since 
$M+w(y) +\varphi(\cdot-y)\geq \varphi_\delta$
and $M+w(y) +\varphi(x^*-y)=\varphi_\delta(x^*)$, by the monotonicity of 
the integral, we get 
$$\widehat\sL\varphi(x^*-y)=\widehat\sL(M+w(y) +\varphi(x^*-y))\leq \widehat\sL\varphi_\delta(x^*)\leq D_1.$$
But this contradicts condition (b) above. Hence $M\leq 0$, proving the claim \eqref{ET1.1B}.
\end{proof}

\subsection{Existence of solution}
We close this section with a remark on the existence of solutions, which
follows from Perron's method together with the comparison principle of
\cite[Theorem~4.1]{KKL19}; see also \cite[Theorem~4.1]{IN}.

Since our results are on interior estimates, it is natural to ask for the existence
of continuous viscosity solutions on a bounded domain $\Omega\subset\Rn$. A
common setting is the solvability of the Dirichlet problem with exterior data
$g\in C(\Omega^{c})$ satisfying suitable growth conditions at infinity, say
$g$ bounded. In general the function produced by Perron's method is
discontinuous, its upper semicontinuous envelope being a subsolution and its
lower semicontinuous envelope a supersolution; the routine procedure is then to
establish a comparison result between a discontinuous viscosity subsolution $u$
and supersolution $v$ of the equation in $\Omega$ satisfying $u\leq v$ in
$\Omega^{c}$.

In the presence of a gradient term, however, attainment of the boundary
condition is not guaranteed for general exterior data $g$, and loss of the
boundary condition may occur; see, for instance,
\cite[Proposition~1]{BCI08} for the case $p=2$. This is a well-known phenomenon
for second-order equations with gradient terms \cite{CDLP}, and boundary layers
for equations of the form \eqref{main} with $p\neq 2$ cannot be ruled out.
Once well-ordered sub- and supersolutions are available up to the boundary, the
comparison principle follows by a contradiction argument together with the
usual doubling of variables. We now describe a case of existence result under particular conditions on the data making the proof simpler.

Throughout this discussion we assume that \hyperlink{H2}{(H2)} holds and that
for every compact set $K$ there exists a modulus of continuity
$\omega_{K}:[0,\infty)\to[0,\infty)$, $\omega_{K}(0)=0$, such that
\begin{equation}\label{unique}
H\Big(y,\frac{x-y}{\varepsilon}\Big)-H\Big(x,\frac{x-y}{\varepsilon}\Big)
\;\leq\;\omega_{K}\Big(|x-y|+\frac{|x-y|^{2}}{\varepsilon}\Big)
\qquad\text{for all } x,y\in K .
\end{equation}
A typical Hamiltonian satisfying \eqref{unique} is $H(x,q)=|q|^{m}+b(x)\cdot q$
with $b$ Lipschitz continuous. For simplicity, we also assume
$p>\frac{2}{2-s}$ and $m>sp$, since the test functions take a simpler form in
this range and Theorem~\ref{Tmain-1} is then available, which shortens the
argument.

Let $\Omega$ be a bounded domain and let $f:\bar\Omega\to\R$ be continuous.
Suppose there exist a subsolution $\underline{u}\in USC(\Rn)$ and a
supersolution $\bar{u}\in C(\Rn)$ of $\sL u=f$ in $\Omega$ with
$\underline{u}\leq\bar{u}$ in $\Rn$. For instance, when $f\geq 0$ and
$H(\cdot,0)\leq 0$ every constant is a subsolution, while a supersolution may
be constructed with the help of Lemma~\ref{L2.2}. To apply Perron's method, set
\[
\mathcal{S}=\big\{\,w:\Rn\to\R \;:\; w\in USC(\Rn),\
w \text{ is a subsolution of } \sL w=f \text{ in } \Omega,\
\text{and } w\leq\bar{u} \text{ in } \Rn \,\big\},
\]
so that $\underline{u}\in\mathcal{S}$. Define
$u(x)=\sup_{w\in\mathcal{S}}w(x)$ for $x\in\Omega$, set $u=\bar{u}$ in
$\Omega^{c}$, and let $u^{*}$ denote the upper semicontinuous envelope of $u$
in $\Omega$, that is, $u^{*}(x)=\limsup_{\Omega\ni x_{k}\to x}u(x_{k})$. Since
$\bar{u}$ is continuous, $\underline{u}\leq u^{*}\leq\bar{u}$ in $\Rn$.

\smallskip
\noindent\emph{Step 1: $u^{*}$ is a subsolution.}
This is routine. Suppose a test function $\varphi$ touches $u^{*}$ from above
at a point $x_{0}$ with $B_{2r}(x_{0})\Subset\Omega$; without loss of
generality the touching may be assumed strict. Then there exist a sequence
$\{w_{k}\}\subset\mathcal{S}$ and points $y_{k}\in B_{r}(x_{0})$ with
$y_{k}\to x_{0}$ such that $\varphi-w_{k}$ attains its minimum over
$B_{r}(x_{0})$ at $y_{k}$ and $w_{k}(y_{k})\to u^{*}(x_{0})$ as $k\to\infty$.
Applying the definition of viscosity subsolution to $w_{k}$ at $y_{k}$ and
using the continuity and monotonicity of the integral operator, one checks that
$\sL\varphi_{r}(x_{0})\leq f(x_{0})$. Hence $u^{*}$ is a subsolution and
$u^{*}\in\mathcal{S}$, which forces $u=u^{*}$. By Theorem~\ref{Tmain-1} we then
have $u\in C(\Omega)$.

\smallskip
\noindent\emph{Step 2: $u$ is a supersolution.}
Suppose not. Then there exist $\varepsilon>0$, a point $x_{0}$ with
$B_{2r}(x_{0})\Subset\Omega$, and a test function $\varphi$ touching $u$ from
below at $x_{0}$ such that $\sL\varphi_{2r}(x_{0})>f(x_{0})+2\varepsilon$. By
\cite[Lemmas~3.8 and 3.9]{KKL19} there exist $t_{0}>0$ and
$\chi\in C_{c}^{\infty}\big(B_{r/2}(x_{0})\big)$ with $\chi\geq 0$ and
$\chi(x_{0})=1$ such that
$\sL(\varphi_{r}+t_{0}\chi)\geq f+\varepsilon$ in $B_{r}(x_{0})$. It is standard
that $\tilde{u}:=\max\{u,\ \varphi_{r}+t_{0}\chi\}$ is then a subsolution of
$\sL\tilde{u}=f$ in $\Omega$; see \cite[Lemma~5.1]{KKL19}. Applying the
comparison principle \cite[Theorem~4.1]{KKL19} on $B_{r}(x_{0})$ yields
$\varphi_{r}+t_{0}\chi\leq\bar{u}$ in $\Rn$, and therefore
$\tilde{u}\in\mathcal{S}$. We remark that the comparison principle of
\cite{KKL19} is stated without the terms $H$ and $f$, but the proof of
Theorem~4.1 there carries over verbatim, thanks to the continuity of $f$ and
\eqref{unique}. Consequently $\tilde{u}\leq u$ in $\Omega$, which fails at
$x_{0}$ by construction. This contradiction shows that $u$ is a supersolution,
and hence a viscosity solution of $\sL u=f$ in $\Omega$.  We recall that the solution above satisfies the Dirichlet boundary condition in its generalized form, see~\cite[Definition 1]{BCI08}. 

\smallskip

We conclude by recording some cases in which the Dirichlet boundary condition
is known to be attained, taking as model the Hamiltonian $H(x,q)=|q|^{m}$. For
$p=2$, the Dirichlet problem is solved in \cite{BCI08} for every continuous
exterior datum $g$ when $m\leq 2s$. We expect analogous existence and
uniqueness results to hold in our setting when $m\leq sp$, this being the
regime in which the diffusion is the leading term and suitable well-ordered
barriers can be constructed. For equations with a dominating gradient term and
$p=2$, solvability is established in \cite{DQT17} for $2s<m<\frac{s}{1-s}$,
under additional assumptions on the right-hand side $f$ and on the exterior
datum $g$. The corresponding question in our setting, while certainly of
interest, lies beyond the scope of the present manuscript.

\section{Proof of Theorem~\ref{Tmain-2}}\label{s-Tmain-2}
\subsection{General strategy}\label{S-genstr}
In this section, we outline the main strategy of the proof, which closely follows \cite{BT25}.
In certain cases, we employ a bootstrap argument. We first establish that 
$u$ is locally $\gamma$-H\"older continuous for $\gamma\in (0,1)$ sufficiently close to 
$1$, and then use this improved regularity to deduce the local Lipschitz continuity of $u$.

Let $1 \leq \varrho_1 < \varrho_2 \leq 2$ be fixed, and consider the doubling function
\begin{equation}\label{E2.2}
\Phi(x, y)= u(x)-u(y)-L\varphi(|x-y|)- \mathrm{d}_1 \psi(x)\quad x, y\in B_2,
\end{equation}
where 
$$
\psi(x) = [(|x|^2-\varrho^2_1)_+]^{\mathrm{d}_2}\quad   x \in B_2,
$$ 
is a {\it localization} function. We choose $\mathrm{d}_2\geq 3$ so that $\psi\in C^2(B_2)$. 
The function $\varphi:[0, 2] \to [0, \infty)$ serves as a {\it regularizing} function and captures the modulus of continuity of $u$. 

In the arguments below, we employ two types of regularizing functions $\varphi$ (after suitable scaling):
\begin{equation}\label{varphi}
\begin{split}
 \varphi_\gamma(t) &=t^{\gamma}\quad \mbox{with} \ \gamma\in (0, 1) \quad \mbox{(for $\gamma$-H\"older profile)}, 
 \\
 \tilde\varphi(t) &=\left\{\begin{array}{ll}
t+\frac{t}{\log t}& \text{for}\; t>0,
\\
0 & \text{for}\; t = 0.
\end{array}
\right.
 \quad  \mbox{(for Lipschitz profile)}.
\end{split}
\end{equation}
Observe that both functions are increasing and concave in a neighbourhood of $t=0$. 

We show that for sufficiently large $\mathrm{d}_1, \mathrm{d}_2$,
and for all $L$ large enough, dependent on $m, n, p, s, A, \norm{f}_{C^{0,1}(B_2)}$, we have $\Phi\leq 0$ in $B_2\times B_2$, which leads to the desired result.

We proceed by contradiction, assuming that $\sup_{B_2\times B_2}\Phi>0$ for all large $L$. Let us choose $\mathrm{d}_1$ large enough, dependent on
$\varrho_2, \varrho_1$ and $\mathrm{d}_2$, so that $\mathrm{d}_1 \psi(x)> 2A$ for all $|x|\geq \frac{\varrho_2+\varrho_1}{2}$. Then for all 
$|x|\geq \frac{\varrho_2+\varrho_1}{2}$, we have $\Phi(x, y)<0$ for all $y\in B_2$. Again, since $\varphi$ is strictly increasing in
$[0, 2]$, if we choose $L$ to satisfy $L\varphi(\frac{\varrho_2-\varrho_1}{4})>2A$, we obtain $\Phi(x, y)<0$ whenever
$|x-y|\geq \frac{\varrho_2-\varrho_1}{4}$. Thus, there exist $\xbar\in B_{\frac{\varrho_2+\varrho_1}{2}}$ and 
$\ybar\in B_{\frac{3\varrho_2}{4}+\frac{\varrho_1}{4}}$ such that
\begin{equation}\label{E2.3}
\sup_{B_2\times B_2}\Phi=\Phi(\xbar,\ybar)>0.
\end{equation}
Denote $\abar=\xbar-\ybar$. From \eqref{E2.3} we have $\abar\neq 0$, and moreover, we have that
\begin{equation}\label{AB03}
L \varphi(|\bar a|) \leq u(\bar x) - u(\bar y) \leq 2A,
\end{equation}
in view of~\eqref{E2.2}. This implies that $|\bar a|$ gets smaller as $L$ enlarges.
Let us denote
$$
\phi(x, y) := L\varphi(|x-y|)+ \mathrm{d}_1 \psi(x).
$$
Note that
\begin{center}
$x\mapsto u(x) - \phi(x, \ybar)$ has a local maximum point at $\xbar$, and\\[2mm]
$y \mapsto u(y) +\phi(\xbar,y)$ has a local minimum point at $\ybar$.
\end{center}
For $\delta \in (0, \frac{\varrho_2-\varrho_1}{4})$ to be chosen later, we define the following test functions

\[
w_1(z)=\left\{\begin{array}{ll}
\phi(z, \ybar) + \kappa_{\xbar} & \text{if}\; z\in B_\delta(\bar x),
\\[2mm]
u(z) & \text{otherwise},
\end{array}
\right.
\quad \text{and}\quad
w_2(z)=\left\{\begin{array}{ll}
-\phi(\xbar, z)+ \kappa_{\ybar} & \text{if}\; z\in B_\delta(\bar y),
\\[2mm]
u(z) & \text{otherwise},
\end{array}
\right.
\]
with $\kappa_{\xbar}=u(\xbar)- \phi(\xbar, \ybar)$, and $\kappa_{\ybar}= u(\ybar) + \phi(\xbar, \ybar)$.

An important point here is that, regardless of the choice of $\varphi$ above, for all sufficiently large
$L$ (depending on $\mathrm{d}_1$ and $\mathrm{d}_2$), we must have $\nabla_x\phi(\xbar,\ybar)\neq 0$ and $\nabla_y\phi(\xbar,\ybar)\neq 0$. Thus, from 
Definition~\ref{Def2.1} we get
$$
\sL w_1(\bar x) \leq f(\xbar)\quad \text{and}\quad 
\sL w_2 (\bar y)   \geq f(\ybar).
$$
As can be seen from \cite{KKL19}, the above principal values are well-defined.
Subtracting the viscosity inequalities at $\bar x$ and $\bar y$, we obtain
\begin{equation}\label{E2.4}
(-\Delta_p)^s w_1(\bar x)- (-\Delta_p)^s w_2(\bar y) + H(\xbar, \grad w_1(\bar x))- H(\ybar, \grad w_2(\bar y))
\leq f(\bar x) - f(\bar y) \leq C |\bar a|.
\end{equation}
At this point we introduce the notation: $J_p(t)=|t|^{p-2}t$ and
\begin{equation}\label{E3.6}
\cI[D] w(x):=  {\rm PV}\int_D J_p(w(x)-w(x+z))\frac{\dz}{|z|^{n+sp}}.
\end{equation}
We also define the following domains
$$\cone=\{z\in B_{\delta_0|\abar|}\; :\; |\langle \abar, z\rangle|\geq (1-\eta_0) |\abar||z|\},
\quad \cD_1=B_\delta \cap \cone^c, \quad \text{and}\quad \cD_2=B_{\tilde\varrho}\setminus (\cD_1\cup\cone),$$
where $\delta_0=\eta_0\in (0, \frac{1}{2})$ will be chosen later, $\tilde\varrho=\frac{1}{4}(\varrho_2-\varrho_1)$, and, in general, $\delta << \delta_0 |\bar a| << \tilde \varrho$. 
 From \eqref{E2.4} and \eqref{E3.6} we arrive at
\begin{align}\label{E2.7}
&\underbrace{\cI[\cone] w_1(\bar x)-\cI[\cone] w_2(\bar y)}_{=I_1} + \underbrace{\cI[\cD_1] w_1(\bar x)-\cI[\cD_1] w_2(\bar y)}_{=I_2} + 
 \underbrace{\cI[\cD_2] w_1(\bar x)-\cI[\cD_2] w_2(\bar y)}_{=I_3}\nonumber
 \\
&\quad + \underbrace{\cI[B^c_{\tilde\varrho}] w_1(\bar x)-\cI[ B^c_{\tilde\varrho}] w_2(\bar y)}_{=I_4} 
 + \underbrace{H(\bar x, \grad w_1(\xbar)) -H(\bar y, \grad w_2(\ybar))}_{=\mathcal{H}}  \leq C |\abar|.
\end{align}
Our main goal is to estimate the terms $I_i$, $i=1,2,3,4$, and $\mathcal{H}$
suitably, so that \eqref{E2.7} leads to a contradiction for all $L$ large
enough. More precisely, we start by assuming $u\in C^{0,\upkappa}(\bar
B_{\varrho_2})$ for some $\upkappa\in[0,1)$, the value $\upkappa=0$
corresponding to a merely continuous function. For a suitably chosen
$\gamma\in(\upkappa,1)$ we show that we can choose
$\delta_0=\eta_0$ small but fixed and the following estimates hold:
\begin{gather*}
I_1\geq CL^{p-1}|\abar|^{\gamma(p-1)-sp},\qquad
I_2\geq -CL^{p-1}\varepsilon_1^{p(1-s)}|\abar|^{\gamma(p-1)-sp},\qquad
I_3\geq -C_{\varepsilon_1}|\abar|^{\gamma(p-1)-sp},
\\[2pt]
|I_4|\leq C\big(|\abar|^{\upkappa}\vee|\abar|^{\upkappa(p-1)}\big),
\qquad\text{and}\qquad
|\mathcal{H}|\leq CL^{m-1}|\abar|^{(m-1)(\gamma-1)+\upkappa\frac{\mathrm{d}_2-1}{\mathrm{d}_2}},
\end{gather*}
where $\delta=\varepsilon_1|\abar|$ with $\varepsilon_1\in(0,\nicefrac{1}{2})$, and the
constant $C$ is independent of $\varepsilon_1$ whereas $C_{\varepsilon_1}$ is
not. Granted these estimates, choosing $\varepsilon_1$ small and then $L$ large
enough, we arrive at
\[
\frac{C}{2}L^{p-1}|\abar|^{\gamma(p-1)-sp}
\;\leq\;
C_1\big(|\abar|^{\upkappa}\vee|\abar|^{\upkappa(p-1)}+|\abar|\big)
+C_1L^{m-1}|\abar|^{(m-1)(\gamma-1)+\upkappa\frac{\mathrm{d}_2-1}{\mathrm{d}_2}},
\]
for constants $C,C_1$ independent of $L$ and $|\abar|$. Since $|\abar|\to 0$ as
$L\to\infty$ by \eqref{AB03}, it then suffices to show that the left-hand side
dominates the right-hand side; this is achieved by choosing
$\gamma\in(\upkappa,1)$ appropriately and then taking $\mathrm{d}_2$ large,
depending on $\upkappa,\gamma,s$ and $p$. The resulting contradiction yields
$u\in C^{0,\gamma}_{\rm loc}$, and we may repeat the argument with $\upkappa$
replaced by $\gamma$. Iterating in this way, we obtain
$u\in C^{0,\gamma}_{\rm loc}$ for every $\gamma\in(0,1)$ whenever
$\frac{sp+1}{p-1}>1$. Finally, the Lipschitz regularity is established by
taking the Lipschitz profile function in \eqref{varphi}, followed by similar estimates.

\subsection{Some key estimates}
In this section we gather a few key estimates from \cite{BS25,BT25}. We start with the estimate of $I_1$
from \cite{BT25}, \cite[Lemma~3.1]{BS25}
\begin{lem}[Estimate of $I_1$]\label{L3.1}
Let $p\in (1, \infty)$.
For $0<|\abar|\leq \frac{1}{8}$, consider the cone 
$\cone=\{z\in B_{\delta_0|\abar|}\; :\; |\langle \abar, z\rangle| \geq (1-\eta_0) |\abar||z|\}$, where $\delta_0=\eta_0\in (0, 1)$.
Then 
\begin{itemize}
\item[(i)] For $\varphi(t)=\varphi_\gamma(t)=t^\gamma$, $\gamma\in (0, 1)$, there exist $L_0, \delta_0$, dependent on $m, \mathrm{d}_1, \gamma$, such that
$$I_1\geq C L^{p-1}|\abar|^{\gamma(p-1)-sp}$$
for all $L\geq L_0$, where the constant $C$ depends on $\delta_0, p, s, \gamma, n$.

\item[(ii)] Let $r_0>0$ be small enough so that for $r\in (0, r_0]$ we have
\begin{gather*}
\frac{r}{2}\leq \tilde\varphi(r)\leq r, \quad \frac{1}{2}\leq \tilde\varphi'(r)\leq 1,
\\
-2(r\log^2(r))^{-1}\leq\tilde\varphi^{\prime\prime}(r)\leq -(r\log^2(r))^{-1}.
\end{gather*}
Letting $\varphi(t) = \tilde \varphi(\frac{r_\circ}{3}t)$,  there exist $L_0, \delta_1$, independent of $\abar$,
such that for $\delta_0=\delta_1(\log^2|\abar|)^{-1}$ we have
$$I_1\geq C L^{p-1} |\abar|^{p-1-sp} (\log^2(|\abar|))^{-\upzeta}$$
for all $L\geq L_0$, where $\upzeta=\frac{n+1}{2}+p-sp$ and the constant $C$ depends only on $\delta_1, p, s, n$.
\end{itemize}
\end{lem}
Next we borrow an estimate of $I_2$ from \cite[Lemmas~3.2 and ~4.2]{BT25}, followed by an intermediate estimate of $I_3$ from \cite[Lemmas~3.3 and ~4.1]{BS25}.
\begin{lem}[Estimate of $I_2$]\label{L3.2}
Let $p\in (1, \infty)$ and $\delta=\varepsilon_1|\abar|$ for $\varepsilon_1\in (0, \nicefrac{1}{2})$. Then there exist $C, L_0$, independent of $\varepsilon_1, |\abar|$, such that
$$I_2 \geq - C L^{p-1} \varepsilon_1^{p(1-s)} (\varphi'(|\abar|))^{p-1} |\abar|^{p(1-s)-1},$$
where $\varphi$ is given by \eqref{varphi}. Moreover, if we set $\delta =\varepsilon_1 (\log^{2\rho}(|\abar|))^{-1}|\abar|$ with
$\rho=\frac{\frac{n+1}{2}+p-sp}{p-sp}$, and let $\varphi(t)=\tilde\varphi(\frac{r_0}{3}t)$ from Lemma~\ref{L3.1}(ii),
then we have $L_0, C>0$ satisfying
$$I_2 \geq - C L^{p-1} \varepsilon_1^{p(1-s)}  |\abar|^{p(1-s)-1} (\log^2(|\abar|))^{-\upzeta},$$
for all $L\geq L_0$.
\end{lem}

\begin{proof}
The proof for $p>2$ follows from \cite[Lemma~3.2]{BT25}. The argument of
\cite[Lemma~4.2]{BT25} contains a minor glitch, however, one of the integrals
there not having been evaluated correctly, so we include a detailed proof for
$p\in(1,2]$.

For a function $g:\Rn\to\R$ we set
\[
\triangle g(x,z)=g(x+z)-g(x)-\grad g(x)\cdot z ,
\]
and we let $\ell(z)=-\grad_x\phi(\xbar,\ybar)\cdot z$. Since
$|\abar|\,\varphi''(|\abar|)\leq\kappa\,\varphi'(|\abar|)$ for all small
$|\abar|$, taking $\delta=\varepsilon_1|\abar|$ gives
(see \cite[Eq.~(3.2)]{BT25})
\begin{equation}\label{AB-01}
\big|\triangle\phi(\cdot,\ybar)(\xbar,z)\big|
\;\leq\; CL|\abar|^{-1}\varphi'(|\abar|)\,|z|^{2}
\qquad\text{for all } z\in B_{\delta},
\end{equation}
where $\varphi$ is given by \eqref{varphi} and the constant $C$ does not depend
on $\delta$, $|\abar|$ or $L$.

We estimate one of the two terms in $I_2$; the other is handled similarly.
Write $\xi=\grad_x\phi(\xbar,\ybar)\neq 0$ and note that, for all large $L$,
\[
\frac{1}{C}\,L\varphi'(|\abar|)\;\leq\;|\xi|\;\leq\;CL\varphi'(|\abar|)
\]
for some constant $C$ independent of $|\abar|$ and $L$. Let
$\hat\xi=\xi/|\xi|$. Since $\ell$ is linear and $\cD_1$ is symmetric about the
origin, $\cI[\cD_1]\ell(0)=0$. Setting
$\tau=CL|\abar|^{-1}\varphi'(|\abar|)$ with $C$ as in \eqref{AB-01}, we obtain
\begin{align*}
\big|\cI[\cD_1]w_1(\xbar)\big|
&=\big|\cI[\cD_1]w_1(\xbar)-\cI[\cD_1]\ell(0)\big|
\\
&\leq C_p\int_{\cD_1}
\big(|\ell(z)|+|\triangle\phi(\cdot,\ybar)(\xbar,z)|\big)^{p-2}
\big|\triangle\phi(\cdot,\ybar)(\xbar,z)\big|\,\frac{\dz}{|z|^{n+sp}}
\\
&\leq C_p\int_{B_{\delta}}
\big(|\ell(z)|+\tau|z|^{2}\big)^{p-2}\,\tau|z|^{2}\,\frac{\dz}{|z|^{n+sp}}
\\
&= C|\xi|^{p-2}\int_{B_{\delta}}
\Big(|\langle\hat\xi,z\rangle|+\frac{\tau}{|\xi|}|z|^{2}\Big)^{p-2}
\tau|z|^{2}\,\frac{\dz}{|z|^{n+sp}}
\\
&\leq C|\abar|^{-1}\big(L\varphi'(|\abar|)\big)^{p-1}
\int_{0}^{\delta}\!\!\int_{\mathbb{S}^{n-1}}
\Big(|\langle\hat\xi,\theta\rangle|+\frac{\tau r}{|\xi|}\Big)^{p-2}
r^{p}\,\frac{{\rm d}\theta\,\dr}{r^{1+sp}}
\\
&\leq C|\abar|^{-1}\big(L\varphi'(|\abar|)\big)^{p-1}
\int_{0}^{\delta}\Big(1+\frac{\tau r}{|\xi|}\Big)^{p-2}
r^{p}\,\frac{\dr}{r^{1+sp}}
\\
&\leq C|\abar|^{-1}\big(L\varphi'(|\abar|)\big)^{p-1}
\int_{0}^{\delta}r^{p-1-sp}\,\dr
\\
&\leq C|\abar|^{-1}\big(L\varphi'(|\abar|)\big)^{p-1}\delta^{p-sp},
\end{align*}
where the third line uses the monotonicity of $b\mapsto(a+b)^{p-2}b$ for
$a,b\geq 0$, the fourth the identity $|\ell(z)|=|\xi|\,|\langle\hat\xi,z\rangle|$,
and the sixth the angular estimate of \cite[Lemma~3.5]{KKL19}; the seventh line
uses $p-2\leq 0$, and the last that $p-sp>0$. The proof is completed by
inserting the appropriate value of $\delta$.
\end{proof}

\begin{lem}[Estimate of $I_3$]\label{L3.3}

Suppose that $u \in C^{0, \upkappa}(\bar B_{\varrho_2})$ and $\vartheta\in (0, 1)$.
\begin{enumerate}

\item[(i)]
Let $p> 2$ and $\upkappa<\min\{1, \frac{sp}{p-2}\}$. There exists $L_0>0$ such that
$$I_3\geq -C \left[\int_\delta^{|\abar|^\vartheta} r^{\upkappa(p-2) + 1-sp} \dr  + 
|\abar|^{\frac{\mathrm{d}_2-1}{\mathrm{d}_2}\upkappa} \int_\delta^{|\abar|^\vartheta} r^{\upkappa(p-2) -sp} \dr + |\abar|^{\upkappa +\vartheta(\upkappa(p-2)-sp)}\right]$$
for all $L\geq L_0$, where the constant $C$ depends on $\upkappa, p, s, n, \mathrm{d}_1, \mathrm{d}_2$ and the $C^{0, \upkappa}$ norm of $u$ in
$\bar{B}_{\varrho_2}$.

\item[(ii)]
Let $p\in (1,2]$. 
There exists $L_0>0$ such that
$$I_3\geq -C \left[\int_\delta^{|\abar|^\vartheta} r^{2(p-1) - 1-sp} \dr  + 
|\abar|^{\frac{\mathrm{d}_2-1}{\mathrm{d}_2}\upkappa(p-1)} \int_\delta^{|\abar|^\vartheta} r^{p-2 -sp} \dr + |\abar|^{\upkappa(p-1) -\vartheta sp}\right]$$
for all $L\geq L_0$, where the constant $C$ depends on $\upkappa, p, s, n, \mathrm{d}_1, \mathrm{d}_2$ and the $C^{0, \upkappa}$ norm of $u$ in
$\bar{B}_{\varrho_2}$.
\end{enumerate}
\end{lem}

We also need an estimate of $I_4$ from \cite[Lemma~3.4]{BS25}
\begin{lem}[Estimate of $I_4$]\label{L3.4}
Let $p\in (1, \infty)$. Suppose that $u\in C^{0, \upkappa}(\bar{B}_{\varrho_2})$ for some $\upkappa\in [0, 1]$. Then
there is a constant $L_0$ such that
$$|I_4|\leq C \max\{|\abar|^\upkappa, |\abar|^{\upkappa(p-1)}\}$$
for all $L\geq L_0$, where the constant $C$ depends on $\tilde\varrho, s, p, n, A$ and the $C^{0, \upkappa}$ norm of $u$ in $\bar{B}_{\varrho_2}$.
\end{lem}

Now we refine the estimate of $I_3$ to a form suitable for our purpose.
\begin{lem}\label{L3.5}
Let  $u \in C^{0, \upkappa}(\bar B_{\varrho_2})$ for some $\upkappa\in [0, 1)$. Then the following hold.
\begin{itemize}
\item[(i)]  For $\varphi=\varphi_\gamma$ (see \eqref{varphi}) with 
$\gamma < \min\{\upkappa + \frac{1}{(p-1)}, 1, \frac{sp}{p-1}\}$ and $\delta=\varepsilon_1 |\abar|$ we have
\begin{equation}\label{I3holder}
I_3 \geq -C_{\varepsilon_1} |\abar|^{\gamma(p-1)-sp}
\end{equation}
for some $C_{\varepsilon_1}>0$ and for all $L\geq L_0$. Here $C_{\varepsilon_1}$ depends on $\varepsilon_1, \gamma, n, s, p, \mathrm{d}_1, \mathrm{d}_2, A$ and $\norm{u}_{C^{0, \upkappa}(B_{\varrho_2})}$.

\item[(ii)] Suppose that $\frac{sp+1}{p-1} > 1$. If $\upkappa$ can be chosen arbitrarily close to $1$, then choosing $\varphi=\tilde\varphi(\frac{r_\circ}{3}\cdot)$ (see Lemma~\ref{L3.1}) and $\delta =\varepsilon_1 (\log^{2\rho}(|\abar|))^{-1}|\abar|$,
we have
\begin{equation}\label{I3lipschitz}
I_3 \geq -C_{\varepsilon_1}  |\abar|^{(p-1)-sp +\epsilon_\circ}
\end{equation}
for some $\epsilon_\circ >0$ and for all $L\geq L_0$, where $C_{\varepsilon_1}$ depends on $\varepsilon_1, n, s, p, \upkappa, \mathrm{d}_1, \mathrm{d}_2, A$ and $\norm{u}_{C^{0, \upkappa}(B_{\bar\varrho_2})}$.
\end{itemize}
\end{lem}

\begin{proof}
For $p\geq 2$, (i) follows from \cite[Proposition~3.5]{BT25} (see equation (3.13) there) and for $p\in (1, 2)$, it can be found in the proof of Theorem~2.1 in Section~4.2 of \cite{BT25} (see page 27 there).

Now consider (ii). First, we suppose that $p\in (2, \infty)$. By our hypothesis, $sp>p-2$. First, we suppose $sp\geq p-1$.
Choose $\upkappa\in (0, 1)$ large enough such that $\upkappa(p-1) + 1 - sp>0$ and $\mathrm{d}_2\geq 3$ large enough so that
$$\frac{\mathrm{d}_2-1}{\mathrm{d}_2}\upkappa+ \upkappa(p-2) + 1 - sp>0.$$
Now set $\vartheta\in (0, 1)$ small enough so that  $\upkappa + \vartheta[\upkappa(p-2)  - sp]>0$. Now it is easily seen that
\begin{align*}
\int_\delta^{|\abar|^\vartheta} r^{\upkappa(p-2) + 1-sp} \dr &\leq C |\abar|^{\vartheta[\upkappa(p-2) + 2-sp]},
\\
|\abar|^{\frac{{\rm d}_2-1}{{\rm d}_2}\upkappa} \int_\delta^{|\abar|^\vartheta} r^{\upkappa(p-2) -sp} \dr &\leq C_{\varepsilon_1} (\log^{2\rho}(|\abar|))^{sp-1-\upkappa(p-2)}
|\abar|^{\frac{{\rm d}_2-1}{{\rm d}_2}\upkappa+ \upkappa(p-2) + 1 - sp},
\end{align*}
where in the last inequality we used  $\upkappa(p-2)-sp\leq \upkappa(p-2) - (p-1)\leq (p-2)(\upkappa-1)-1<-1$. Thus, by Lemma~\ref{L3.3}(i), 
and our choice of parameters, we can find $\epsilon_\circ$ satisfying
$I_3\geq -C_{\varepsilon_1} |\abar|^{\epsilon_\circ}
\geq -C_{\varepsilon_1} |\abar|^{(p-1)-sp+\epsilon_\circ}$.

Next we suppose $sp<p-1$ and choose $\upkappa\in (\frac{sp}{p-1}, 1)$. Define $\vartheta=\frac{(p-1)-sp + \epsilon}{\upkappa(p-2)+2-sp} $. Since $\upkappa(p-2)+2-sp\geq 2-\upkappa>1$ and $(p-1)-sp < (p-1)-(p-2)=1$
we have $\vartheta\in (0, 1)$ for $\epsilon$ small enough. We use this $\vartheta$ in Lemma~\ref{L3.3}.
It is then easy to see that
\begin{align*}
\int_{\delta}^{|\abar|^{\vartheta}} r^{\upkappa(p-2) +1-sp} \dr &\leq \frac{1}{\upkappa(p-2)+2-sp} |\abar|^{\vartheta(\upkappa(p-2)-sp+2)}
\\
&\leq  |\abar|^{(p-1)-sp +  \epsilon}
\end{align*}
using the fact that $\upkappa(p-2)+2-sp>1$. It can be easily checked that
$$
\upkappa + \vartheta[\upkappa(p-2)-sp]> (p-1)-sp 
\Leftrightarrow 
1 < \frac{1}{2(p-1)} \left[\upkappa(\upkappa(p-2)+2-sp)+2sp \right] +\tilde\epsilon,
$$
where $\tilde\epsilon=\epsilon\frac{\upkappa(p-2)-sp}{2(p-1)} > -\epsilon\frac{sp}{2(p-1)}$. Now, we
note that
\begin{align*}
\frac{1}{2(p-1)} \left[\upkappa(\upkappa(p-2)+2-sp)+2sp \right]-\upkappa
&=  \frac{\ell(\upkappa)}{2(p-1)},
\end{align*}
where 
\begin{equation}\label{ell}
\ell(y)= y^2(p-2) -[sp+2(p-2)]y + 2sp.
\end{equation}
We have $\ell(\frac{sp}{p-2})=0$, $\ell(1) =sp-(p-2) >0$ and $\ell'(1)=-sp<0$ . 
Thus the function $\ell$ is strictly decreasing and positive in $(-\infty, 1]$.
Let 
\begin{equation}\label{rho}
\uprho =\inf_{t \in (0,1)}\frac{\ell(t)}{4(p-1)}=\frac{\ell(1)}{4(p-1)} > 0.
\end{equation}
We choose $\upkappa$ close to $1$ so that $\upkappa+\uprho>1$, and then for $\epsilon$ small enough such that $\uprho +\tilde \epsilon > \uprho -\epsilon \frac{sp}{2(p-1)} >0$ we get
\begin{equation}\label{here}
(p-1)-sp<  \upkappa + \vartheta[\upkappa(p-2)-sp].
\end{equation}
Since $|\abar| < 1$, we  have
$$|\abar|^{\upkappa+\vartheta(\upkappa(p-2)-sp)} \leq |\abar|^{(p-1) -sp +\varepsilon_2}$$
for some $\varepsilon_2>0$.
Set $\mathrm{d}_2$ large enough such that $\frac{\upkappa}{\mathrm{d}_2} < \vartheta$ which would imply $\upkappa\frac{\mathrm{d}_2-1}{\mathrm{d}_2} + \vartheta(\upkappa(p-2)-sp+1) > \upkappa + \vartheta(\upkappa(p-2)-sp)$.
Therefore, from the above estimate, we have
\begin{align*}
|\abar|^{\frac{\mathrm{d}_2-1}{\mathrm{d}_2}\upkappa } \int_{\delta}^{|\abar|^{\vartheta}} r^{\upkappa(p-2)-sp} \dr &\leq \frac{1}{\upkappa(p-2)-sp+1} |\abar|^{\upkappa \frac{\mathrm{d}_2-1}{\mathrm{d}_2} + \vartheta (\upkappa(p-2)-sp+1)}
\\
&\leq C |\abar|^{(p-1)-sp + \varepsilon_2},
\end{align*}
using the fact $\upkappa(p-2)-sp+1>0$.
Combining the above three estimates in Lemma~\ref{L3.3}(i) we have \eqref{I3lipschitz} for $p>2$.

Next, we consider $p\in (1, 2]$. Choose $\upkappa\in (\frac{p-1}{p}, 1)$ large enough such that 
$\frac{1}{4}sp + \upkappa(p-1)>(p-1)$.
Now choose $\mathrm{d}_2\geq 3$ large enough so that 
$$\frac{\mathrm{d}_2-1}{\mathrm{d}_2}\upkappa(p-1)>(p-1)-sp.$$
Set $\vartheta\in (\frac{1}{2}, \frac{3}{4})$. We again, evaluate the terms in Lemma~\ref{L3.3}(ii).
From our choice of parameters
\begin{align*}
&|\abar|^{\frac{\mathrm{d}_2-1}{\mathrm{d}_2}\upkappa(p-1)} \int_\delta^{|\abar|^{\vartheta}} r^{p-2-sp} \dr 
\\
&\qquad \leq C |\abar|^{\frac{\mathrm{d}_2-1}{\mathrm{d}_2} \upkappa(p-1)} [ |\log(\delta)|1_{\{sp=p-1\}} + \delta^{p-1-sp}1_{\{sp>p-1\}} + |\abar|^{\vartheta(p-1-sp)}1_{\{p-1>sp\}}].
\end{align*}
From the definition of $\delta$, given by (ii), and the fact that $|\abar|\to 0$ as $L\to \infty$, we can find $L_0, \varepsilon_3>0$ so that 
$$|\abar|^{\frac{\mathrm{d}_2-1}{\mathrm{d}_2}\upkappa(p-1)} \int_\delta^{|\abar|^{\vartheta}} r^{p-2-sp} \dr\leq C_{\varepsilon_1} |\abar|^{p-1-sp+ \varepsilon_3}$$
for $L\geq L_0$, where the constant $C_{\varepsilon_1}$ depends on $\varepsilon_1$. On the other hand, if $2(p-1)-sp>0$, we obtain
$$\int_\delta^{|\abar|^\vartheta} r^{2(p-1) - 1-sp} \dr\leq C |\abar|^{\vartheta(2(p-1)-sp)}
\leq C |\abar|^{p-1-sp+\varepsilon_4},$$
as $\vartheta(2(p-1)-sp)>2\vartheta (p-1) -sp>p-1-sp$,
and if $2(p-1)-sp\leq 0$, we have $(p-1)-sp + \frac{p-1}{2}< 0$, giving us
$$\int_\delta^{|\abar|^\vartheta} r^{2(p-1) - 1-sp} \dr\leq \int_\delta^{1} r^{(p-1)+\frac{p-1}{2} - 1-sp} \dr\leq C_{\varepsilon_1}|\abar|^{p-1-sp +\varepsilon_4}$$
for some positive $\varepsilon_4$ and $L\geq L_0$, provided we choose $L_0$ large enough. Thus, combining these estimates
in Lemma~\ref{L3.3} together with the fact that 
$$\upkappa(p-1)-\vartheta sp\geq \upkappa(p-1)-\frac{3}{4} sp> (p-1)-sp,$$
which we use to estimate the $|\abar|^{\upkappa(p-1)-\vartheta sp}$,
we have \eqref{I3lipschitz}.
\end{proof}

\subsection{Proofs of Theorems~\ref{Tmain-2} and ~\ref{T1.3}}
Now we can provide proofs of Theorems~\ref{Tmain-2} and ~\ref{T1.3}. First, we estimate $\mathcal{H}$ using \hyperlink{H1}{(H1)}.
Suppose that $u\in C^{0, \upkappa}(\bar{B}_{\varrho_2})$ for some $\upkappa\in [0, 1)$. The case $\upkappa=0$ should be understood 
as $u\in C(\bar{B}_{\varrho_2})$. Letting $\varphi=\varphi_\gamma$ in \eqref{AB03} we get that
\begin{equation}\label{ET1.2A}
L|\abar|^\gamma\leq \norm{u}_{C^{0, \upkappa}(\bar{B}_{\varrho_2})} |\abar|^\upkappa \Rightarrow
L |\abar|^{\gamma-\upkappa}\leq \norm{u}_{C^{0, \upkappa}(\bar{B}_{\varrho_2})}.
\end{equation}
Again, since $\Phi(\xbar,\ybar)>0$, we obtain $\mathrm{d}_1\psi(\xbar)\leq u(\xbar)-u(\ybar)$. From the definition of $\psi$ we then get
$$
(|\xbar|^2-\varrho^2_1)_+\leq \mathrm{d}_1^{-\frac{1}{\mathrm{d}_2}} \left(\norm{u}_{C^{0, \upkappa}(\bar{B}_{\varrho_2})} |\abar|^\upkappa\right)^{\frac{1}{\mathrm{d}_2}},
$$
leading to
\begin{equation}\label{ET1.2B}
|\grad \psi(\xbar)|\leq 4 \mathrm{d}_2 [(|\xbar|^2-\varrho^2_1)_+]^{\mathrm{d}_2-1}\leq \kappa |\abar|^{\frac{\upkappa(\mathrm{d}_2-1)}{\mathrm{d}_2}}
\end{equation}
for some constant $\kappa$, dependent on $\mathrm{d}_1, \mathrm{d}_2$ and $\norm{u}_{C^{0, \upkappa}(\bar{B}_{\varrho_2})}$. Thus, using
\hyperlink{H1}{(H1)} we get
\begin{align}
|\mathcal{H}| &=\left| H(\xbar, \grad_x\phi(\xbar,\ybar)) - H(\ybar, -\grad_y\phi(\xbar,\ybar)) \right| \nonumber
\\
&\leq C_{H,2}\left[ |\abar| (1+ |\grad_y\phi(\xbar,\ybar)|^m) + \mathrm{d}_1|\grad\psi(\xbar)| (|\grad_y\phi(\xbar,\ybar)|^{m-1}
+ |\grad\psi(\xbar)|^{m-1} +1 ) \right] \nonumber
\\
&\leq \kappa_1  \left[|\abar| L^m |\abar|^{m(\gamma-1)} +  |\abar|^{\frac{\upkappa(\mathrm{d}_2-1)}{\mathrm{d}_2}} 
L^{m-1} |\abar|^{(m-1)(\gamma-1)}\right] \nonumber
\\
&=  \kappa_1  \left[L^{m-1} |\abar|^{(m-1)(\gamma-1)} L |\abar|^\gamma  +  |\abar|^{\frac{\upkappa(\mathrm{d}_2-1)}{\mathrm{d}_2}} 
L^{m-1} |\abar|^{(m-1)(\gamma-1)}\right] \nonumber
\\
&\leq \kappa_2 L^{m-1} |\abar|^{(m-1)(\gamma-1)} |\bar a|^{\upkappa \frac{\mathrm{d}_2-1}{\mathrm{d}_2}} \label{Hestim},
\end{align}
where in the third line we use \eqref{ET1.2B} and the fact $|\grad_y\phi(\xbar,\ybar)|\leq L |\abar|^{\gamma-1}$, and in the fifth line
we use \eqref{ET1.2A}. It is also useful to note that the estimate works with the Lipschitz profile function (see \eqref{varphi})
and in this case, \eqref{Hestim} holds with $\gamma=1$.

We first need the following lemma.
\begin{lem}\label{L3.6}
Assume $1< m \leq sp$, $f\in C^{0,1}(\bar{B}_2)$ and let $u$ be a viscosity solution to~\eqref{main}. We let $H$ satisfy
\hyperlink{H1}{(H1)}.
Then
\begin{itemize}
\item[$(i)$] for $p-2 < sp$, $u \in C^{0, \gamma}_{\rm loc}(B_2)$ for every $\gamma \in (0,1)$. 

\item[$(ii)$] for $p -2 \geq sp$, $u \in C^{0, \gamma}_{\rm loc}(B_2)$ for every $\gamma \in (0, \frac{sp - m + 1}{p - m - 1})$.
\end{itemize}
\end{lem}

\begin{proof}
Assume $u\in C^{0, \upkappa}(\bar{B}_{\varrho_2})$ for 
some $\upkappa\in[0, 1)$, starting with the case $\upkappa=0$, and we provide an iterative process to get the expected estimate. Take $\gamma \in (0, \min \{1, \frac{sp}{p-1}, \frac{1}{p-1} \})$, $\varphi_\gamma$ as in~\eqref{varphi}, and use this function in~\eqref{E2.2}. Then, by Lemmas~\ref{L3.1}, \ref{L3.2} and ~\ref{L3.5} we can first choose $\varepsilon_1$ suitably small and then $L_0$ large enough, dependent on $\varepsilon_1$, in order to get
\begin{equation}\label{here1} 
I_1+I_2+I_3\geq C L^{p-1}|\abar|^{\gamma(p-1)-sp},
\end{equation}
for all $L\geq L_0$. Using this, together with Lemma~\ref{L3.4} and the estimates \eqref{Hestim} in \eqref{E2.7} we arrive at
$$  L^{p-1}|\abar|^{\gamma(p-1)-sp}\leq  C (|\abar|^\upkappa\vee |\abar|^{\upkappa(p-1)}+|\abar|) + C L^{m-1} |\abar|^{(m-1)(\gamma-1)} |\bar a|^{\upkappa \frac{\mathrm{d}_2-1}{\mathrm{d}_2}}$$
for all $L\geq L_0$. 

Last inequality leads to 
\begin{equation}\label{EL3.6A}
  1\leq  C L^{1-p}(|\abar|^\upkappa\vee |\abar|^{\upkappa(p-1)} +|\abar|) + C L^{m-p} |\abar|^{\gamma(m-p)+sp-m+1} |\bar a|^{\upkappa \frac{\mathrm{d}_2-1}{\mathrm{d}_2}}.
 \end{equation}
 
 At this point, we introduce the notation
$$
\gamma_0 =\frac{sp-m+1}{p-m} > 0.
$$

Then, taking $\gamma < \min \{ 1, \frac{sp}{p-1}, \frac{1}{p - 1}, \gamma_0 \}$, we conclude that
\begin{equation*}
  1\leq  C L^{1-p}(|\abar|^\upkappa\vee |\abar|^{\upkappa(p-1)} +|\abar|) + C L^{m-p}, 
 \end{equation*}
for some $C > 0$ depending on the data and $A$.  Notice that the RHS of the above display tends to zero as $L$ enlarges. 
This is a contradiction for all $L\geq L_0$ if we set $L_0$ large enough that violates the above inequality. Therefore, $\Phi\leq 0$ in $B_2\times B_2$. Since $\psi=0$ in $B_{\varrho_1}$, it implies that 
$$ 
|u(x)-u(y)|\leq L_0 |x-y|^{\gamma},
$$
with $\gamma <  \min \{1, \frac{1}{p - 1}, \frac{sp}{p - 1}, \gamma_0\}$.

Now, fix $\upkappa <  \min \{1, \frac{1}{p - 1}, \frac{sp}{p - 1}, \gamma_0\}$ and without loss of generality we assume $u \in C^{0,\upkappa}_{\rm loc}(B_2)$. Following the same procedure as above and taking 
$\gamma = \min \{1, \upkappa + \frac{1}{2(p - 1)}, \frac{sp}{p - 1}, \gamma_0\}$ we arrive at~\eqref{EL3.6A}, from which we conclude a contradiction for $L$ large enough. 
Applying a bootstrapping argument we see
that $u$ is locally $\tilde \gamma$-H\"{o}lder for any $\tilde \gamma \in (0, \min \{ 1, \frac{sp}{p-1}, \gamma_0 \})$. Using a simple covering argument we see that $u\in C^{0, \tilde \gamma}_{\rm loc}(B_2)$. 

Now, if $p - 1 \leq sp$ we have $\gamma_0\wedge \frac{sp}{p - 1} \geq 1$ and therefore, there is no obstruction to continuing the iteration until we reach
any exponent $\gamma < 1$, whence 
$u \in C^{0, \gamma}_{\rm loc}(B_2)$ for every $\gamma \in (0,1)$ and $(i)$ follows in this sub-case.

\smallskip

Now we deal with the case $sp < p - 1$. We note that $\gamma_0 < \frac{sp}{p-1} < 1$. Therefore, letting $\upkappa$ close to $\gamma_0$ so that $\upkappa+\frac{1}{2(p-1)}>\gamma_0$ and using the argument of the first part
we see that $u \in C^{0, \gamma_0}_{\rm loc}(B_2)$.  Notice that in this case we have $p - m > 1$ and the geometric series with ratio $(p - m)^{-1}$ converges. 
Define
\begin{equation}\label{gamma*}
\gamma^* := \gamma_0 \sum_{i=0}^\infty \frac{1}{(p - m)^i}=\frac{sp-m+1}{p-m-1}.
\end{equation}
At this point, we first consider the sub-case $p - 2 < sp < p-1$ (this also implies $p>2$ as $1<m\leq sp$). Thus
\begin{align*}
\gamma^*  > 1,
\end{align*} 
and we can take $\alpha \in (0,1)$ small enough in order to have
\begin{align*}
\gamma_0 \sum_{i=0}^\infty \frac{1}{(p - m + \alpha)^{i}} > 1.
\end{align*}
With this, we define
$$
\gamma_k = \gamma_0 \sum_{i=0}^{k}(p - m +\alpha)^{-i}.
$$

Denote by $\bar k \in \mathbb N\cup\{0\}$ the largest number for which $\gamma_{\bar k} \leq \frac{sp}{p - 1}$.
We show that $u\in C^{0, \gamma_{\bar k}}_{\rm loc}(B_2)$.
If $\bar{k}=0$, then $\gamma_{\bar k}=\gamma_0$ and the claim is proved above. So, we let $\bar{k}\geq 1$.
We prove that for each $k \leq \bar k - 1$, if $u \in C^{0, \gamma_k}_{\rm loc}(B_2)$, then 
$u \in C^{0, \gamma_{k + 1}}_{\rm loc}(B_2)$. In fact, following exactly the steps of the first part of the proof with $\gamma= \gamma_{k + 1}$ and $\gamma_k = \upkappa$, inequality~\eqref{EL3.6A} takes the form
\begin{equation}\label{shoulder}
  1\leq  C L^{1-p}(|\abar|^{\gamma_k} +|\abar|) + C L^{m-p} |\abar|^{\gamma_{k+1}(m-p)+sp-m+1} |\bar a|^{\gamma_k \frac{\mathrm{d}_2-1}{\mathrm{d}_2}}.
\end{equation} 

Using that $\gamma_{k + 1} = \gamma_0 + \frac{\gamma_k}{p - m + \alpha}$, we see that the exponent of $|\bar a|$ in the last term of the right-hand side can be written as
\begin{align*}
& \gamma_{k+1}(m-p)+sp-m+1 + \gamma_k \frac{\mathrm{d}_2-1}{\mathrm{d}_2} \\
& =  (p - m) \Big{(} -\gamma_{k + 1}  + \gamma_0 + \gamma_k \frac{\mathrm{d}_2 - 1}{\mathrm{d}_2} \frac{1}{p - m} \Big{)} \\
 &  = (p - m) \gamma_k \Big{(} -\frac{1}{p - m +\alpha} + \frac{\mathrm{d}_2 - 1}{\mathrm{d}_2} \frac{1}{p - m} \big{)},
\end{align*}
and therefore, taking $\mathrm{d}_2$ large enough in terms of $p,m$ and $\alpha$, but not on $k$, we conclude that the exponent is nonnegative. Thus, we arrive at
$$
  1\leq  C L^{1-p}(|\abar|^{\gamma_k} +|\abar|) + C L^{m-p}, 
$$
and we reach a contradiction by taking $L$ large enough. Since the argument goes through by taking 
$\gamma = \frac{sp}{p-1}$,
when we know that $u \in C^{0, \gamma_{\bar k}}_{\rm loc}(B_2)$, we obtain that $u \in C^{0, \frac{sp}{p - 1}}_{\rm loc}(B_2)$, see~\cite[Theorem 2.1]{BT25}.

To improve the regularity further we cannot rely on Lemma~\ref{L3.5}-$(i)$. 
{ Fix $\tilde\gamma\in (\upkappa, 1)$ and $\gamma\in (\upkappa, \tilde\gamma]$}. Using Lemma~\ref{L3.3}(i), for every $\upkappa \in [sp/(p-1), 1)$ and $\vartheta \in (0,1)$, we may write the estimate
$$
I_3 \geq - C (|\bar a|^{\beta_1} + |\bar a|^{\beta_2} + |\bar a|^{\beta_3} ),
$$
where $\beta_i, i=1,2,3$, are given by
\begin{equation*}
\beta_1 = \vartheta (\upkappa(p - 2) - sp + 2), \quad \beta_2 = \frac{\mathrm{d}_2 - 1}{\mathrm{d}_2} \upkappa + \vartheta(\upkappa(p - 2) -sp + 1), \quad \beta_3 = \upkappa + \vartheta(\upkappa(p-2) -sp).
\end{equation*}
Since $\upkappa(p-2) -sp +\upkappa\geq 0$ for $\upkappa\geq \frac{sp}{p-1}$, 
we have $\beta_i>0$ for $i=1,2,3$. We choose $\vartheta=\frac{\gamma(p-1)-sp+\epsilon}{\upkappa(p-2)+2-sp}$. For small enough $\epsilon$ we have $\vartheta\in (0, 1)$. This is possible because $\gamma(p-1)>sp$, 
$\upkappa(p-2)+2-sp=\upkappa(p-1)-sp+2-\upkappa>0$ and
\begin{align*}
 \gamma(p-1)<\gamma(p-2) +\gamma <(p-2)+\gamma<sp +\gamma
 \leq \upkappa(p-1) +\gamma<\upkappa(p-2) + 2.
\end{align*}
It is then easily seen that $\beta_1>\gamma(p-1)-sp$. Also, if we choose
$\epsilon\in (0,1)$ small enough the arguments of Lemma~\ref{L3.5}(ii) (see \eqref{here}) give
$$\gamma(p-1)-sp<  \upkappa + \vartheta[\upkappa(p-2)-sp].$$
Hence we have $\beta_2, \beta_3>\gamma(p-1)-sp$, provided we choose $\mathrm{d}_2$ suitably large and $\gamma\leq \upkappa +\uprho$, where $\uprho$ is given by \eqref{rho}. 
Though Lemma~\ref{L3.5}(ii) is done for $\gamma=1$, the similar arguments work here as well since we are keeping $\gamma-\upkappa\leq \uprho$ and the definition of $\vartheta$ modified accordingly.
Therefore, \eqref{here1} holds for all large $L$, and \eqref{shoulder} takes the form
\begin{equation}\label{here2}
1\leq  C L^{1-p}(|\abar|^{\upkappa} +|\abar|)|\abar|^{sp-\gamma(p-1)} + C L^{m-p} |\abar|^{\gamma(m-p)+sp-m+1} |\bar a|^{\upkappa \frac{\mathrm{d}_2-1}{\mathrm{d}_2}}.
\end{equation}
At this point, we note that if $\gamma<\gamma_0 + \frac{\upkappa}{p-m}$ the exponent of $|\abar|$ to the rightmost term
can be made positive, provided we choose $\mathrm{d}_2$ large, but we also must take care of the exponents in the first two terms as $sp-\gamma(p-1)<0$.
Given any $\tilde\gamma\in (0, 1)$, if we let $\gamma\leq  \min \{\upkappa +\frac{sp-\tilde\gamma(p-2)}{p-1}, 
\upkappa + \frac{sp-p+2}{p-m}, \upkappa+\uprho\}$, we have
$$
\gamma(p-1)\leq \upkappa(p-1)+sp-\tilde\gamma(p-2)
\leq \upkappa + sp -(\tilde\gamma-\upkappa)(p-2)\leq \upkappa + sp
$$
and
\begin{align*}
\gamma (m-p)+ sp-m+1 + \upkappa &\geq (m-p)\upkappa - (sp-p+2)
+ sp-m+1+\upkappa
\\
&=(1-\upkappa)(p-1-m)>0,
\end{align*}
as $m\leq sp<p-1$. Therefore, for large enough $\mathrm{d}_2$ we have the exponent of $|\abar|$ to the right most term of \eqref{here2} positive. With this choice of $\gamma$ we get a contradiction from \eqref{here2} as we enlarge $L$, giving us
$C^{0, \gamma}_{\rm loc}(B_2)$ regularity. Now we can bootstrap the above iteration to conclude $C^{0, \tilde\gamma}_{\rm loc}(B_2)$ regularity. This completes the proof of (i).




\smallskip

For (ii) we assume $sp\leq p-2$. Since $1<m\leq sp\leq p-2$, this forces
$p>3$. The series defining $\gamma^{*}$ in \eqref{gamma*} is still convergent,
since $m\leq sp\leq p-2$ gives $p-m\geq 2$, and
\[
\gamma^{*}=\frac{sp-m+1}{p-m-1}\;\leq\;\frac{sp}{p-2}\;\leq\;1 .
\]
If $\gamma^{*}\leq\frac{sp}{p-1}$, the iteration of part~(i)
(the arguments after \eqref{shoulder})  yields
$u\in C^{0,\gamma}_{\rm loc}(B_{2})$ for every $\gamma\in(0,\gamma^{*})$ and
there is nothing further to prove; we may therefore assume
$\gamma^{*}>\frac{sp}{p-1}$. Since $u\in C^{0,\gamma_{0}}_{\rm loc}(B_{2})$ by
the first part of this theorem, we may then repeat the inductive argument
leading to \eqref{shoulder} for all $k$ and obtain, as before,
$u\in C^{0,\frac{sp}{p-1}}_{\rm loc}(B_{2})$.

Fix $\tilde\gamma\in\big[\tfrac{sp}{p-1},\gamma^{*}\big)$ and let
$\upkappa\in\big[\tfrac{sp}{p-1},\tilde\gamma\big)$, and set
\[
\varpi:=(\gamma^{*}-\tilde\gamma)\Big(1-\frac{1}{p-m}\Big)>0 .
\]
We recall the quadratic function $\ell$ from \eqref{ell} in the proof of
Lemma~\ref{L3.5}(ii). Since $\ell\big(\tfrac{sp}{p-2}\big)=0$,
$\ell(1)=sp-(p-2)\leq 0$ and $\ell'(1)=-sp<0$, the function $\ell$ is positive
and decreasing on $\big(-\infty,\tfrac{sp}{p-2}\big)$; as
$\tilde\gamma<\gamma^{*}\leq\frac{sp}{p-2}$, the quantity
\[
\uprho_{\tilde\gamma}
=\frac{1}{4(p-1)}\min_{y\in[0,\tilde\gamma]}\ell(y)
=\frac{\ell(\tilde\gamma)}{4(p-1)}
\]
is therefore strictly positive. As in part~(i), we appeal to
Lemma~\ref{L3.3}(i), with $\beta_{i}$, $i=1,2,3$, as defined above. Set
\[
\gamma=\min\left\{\tilde\gamma,\;
\upkappa+\frac{1}{2(p-1)},\;
\upkappa+\frac{sp-(p-2)\tilde\gamma}{2(p-1)},\;
\upkappa+\uprho_{\tilde\gamma},\;
\upkappa+\frac{\varpi}{2}\right\},
\qquad
\vartheta=\frac{\gamma(p-1)-sp}{\upkappa(p-2)+2-sp},
\]
and note that $\gamma>\upkappa$, each of the five quantities above exceeding
$\upkappa$. Since $\gamma>\upkappa\geq\frac{sp}{p-1}$ we have
\begin{equation}\label{0001}
0\leq\gamma(p-1)-sp<\upkappa(p-1)-\gamma(p-2)\leq\upkappa(p-1)-\upkappa(p-2)=\upkappa,
\end{equation}
the middle inequality following from the third term in the definition of
$\gamma$, while $\upkappa(p-2)+2-sp\geq 2-\upkappa>1$; as $\upkappa<1$, this
gives $\vartheta\in(0,1)$. By the choice of $\vartheta$ we have
$\beta_{1}=\gamma(p-1)-sp$. For $\beta_{2}$ and $\beta_{3}$ it suffices to
bound the smaller quantity
$\frac{\mathrm{d}_{2}-1}{\mathrm{d}_{2}}\upkappa+\vartheta[\upkappa(p-2)-sp]$
from below, and we note that
\[
\frac{\mathrm{d}_{2}-1}{\mathrm{d}_{2}}\upkappa+\vartheta\big[\upkappa(p-2)-sp\big]
>\gamma(p-1)-sp
\iff
\gamma<\frac{1}{p-1}\left[\frac{\mathrm{d}_{2}-1}{2\mathrm{d}_{2}}
\upkappa\big(\upkappa(p-2)+2-sp\big)+sp\right],
\]
together with
\begin{align*}
\frac{1}{p-1}\left[\frac{\mathrm{d}_{2}-1}{2\mathrm{d}_{2}}
\upkappa\big(\upkappa(p-2)+2-sp\big)+sp\right]-\upkappa
&=\frac{\ell(\upkappa)}{2(p-1)}
-\frac{\upkappa\big(\upkappa(p-2)+2-sp\big)}{2\mathrm{d}_{2}(p-1)}
\\
&\geq 2\uprho_{\tilde\gamma}-\frac{p(1-s)}{2\mathrm{d}_{2}(p-1)}
\;>\;\uprho_{\tilde\gamma},
\end{align*}
provided $\mathrm{d}_{2}$ is chosen large enough. Since
$\gamma\leq\upkappa+\uprho_{\tilde\gamma}$, this choice of $\mathrm{d}_{2}$
gives $\beta_{2},\beta_{3}>\gamma(p-1)-sp$. Following the arguments of
part~(i), we therefore arrive at \eqref{here2}, in which the exponent of
$|\abar|$ in the first product on the right-hand side is positive by our choice
of $\gamma$, see \eqref{0001}.
To treat the last term of \eqref{here2}, observe first that
$\gamma_{0}=\gamma^{*}\big(1-\frac{1}{p-m}\big)$, whence
\[
\gamma_{0}-\upkappa\Big(1-\frac{1}{p-m}\Big)
=(\gamma^{*}-\upkappa)\Big(1-\frac{1}{p-m}\Big)>\varpi .
\]
Consequently,
\begin{align*}
\gamma(m-p)+sp-m+1+\upkappa\frac{\mathrm{d}_{2}-1}{\mathrm{d}_{2}}
&=(p-m)\left[-\gamma+\gamma_{0}
+\frac{\upkappa}{p-m}\frac{\mathrm{d}_{2}-1}{\mathrm{d}_{2}}\right]
\\
&\geq (p-m)\left[-\upkappa-\frac{\varpi}{2}+\gamma_{0}
+\frac{\upkappa}{p-m}\frac{\mathrm{d}_{2}-1}{\mathrm{d}_{2}}\right]
\\
&\geq (p-m)\left[\frac{\varpi}{2}
-\frac{\tilde\gamma}{(p-m)\mathrm{d}_{2}}\right]>0,
\end{align*}
provided $\mathrm{d}_{2}$ is further enlarged, depending on $\varpi$. We may
thus derive a contradiction from \eqref{here2}, which gives
$C^{0,\gamma}_{\rm loc}$ regularity. Repeating the above iteration completes
the proof of (ii).
\end{proof}


\begin{proof}[Proof of Theorem~\ref{T1.3}]
The proof follows from Lemma~\ref{L3.6}(ii).
\end{proof}

Now we focus on the proof of Theorem~\ref{Tmain-2}.
\begin{proof}[Proof of Theorem~\ref{Tmain-2}]
We break the proof of Theorem~\ref{Tmain-2} in several cases. In the proof below, we impose the blanket assumption $\frac{sp+1}{p-1}>1$.
\medskip

\noindent{\bf Case 1.} ($sp<p-1$ and $m>1$).
Since $sp<p-1<m$ is covered by Theorem~\ref{Tmain-1}, we only consider the case $1< m\leq p-1$. In fact, the proof below works
under the conditions $1<m\leq p-1$ and $\frac{sp+1}{p-1}>1$.
From Theorem~\ref{Tmain-1} and Lemma~\ref{L3.6} we see that $u\in C^{0, \upkappa}_{\rm loc}(B_2)$ for any $\upkappa<1$. 
We set $\upkappa$ close to $1$ and $\mathrm{d}_2$ large enough such that $\upkappa\frac{\mathrm{d}_2-1}{\mathrm{d}_2} + sp -(p-1) >0$, and $\frac{sp}{p-1} > 1-\upkappa$.

In \eqref{E2.2} we take $\varphi=\tilde\varphi(\frac{r_\circ}{3}\cdot)$, the Lipschitz profile function in \eqref{varphi}. As before, we would like to show that
$\Phi\leq 0$ in $B_2\times B_2$ for some large $L$ dependent on $H, n, s, p$ and $A$. We argue by contradiction, assuming \eqref{E2.3}, and 
arrive at \eqref{E2.7}. Choosing $\varepsilon_1$ suitably small, we get from Lemmas~\ref{L3.1}-\ref{L3.2} that
$$ I_1+I_2\geq C L^{p-1} |\abar|^{p-1-sp} (\log^2(|\abar|))^{-\upzeta}$$
for all $L\geq L_0$, where $\upzeta=\frac{n+1}{2}+p-sp$. Now using Lemmas~\ref{L3.4}, ~\ref{L3.5} and the estimate \eqref{Hestim} with the Lipschitz profile function we obtain from \eqref{E2.7} that
$$L^{p-1} |\abar|^{p-1-sp} (\log^2(|\abar|))^{-\upzeta}
\leq C \left[|\abar|^{p-1-sp+\epsilon_\circ} + |\abar|^{\upkappa}+ |\abar|^{\upkappa(p-1)} + |\abar| + L^{m-1} |\bar a|^{\upkappa \frac{\mathrm{d}_2-1}{\mathrm{d}_2}} \right]
$$ 
for all $L\geq L_0$. This implies
\begin{align*}
 1&\leq C L^{1-p}\left[|\abar|^{\epsilon_\circ} + |\abar|^{\upkappa+sp-(p-1)} + 
|\abar|^{(\upkappa-1)(p-1)+sp} + |\abar|^{sp-(p-2)}\right] (\log^2(|\abar|))^{\upzeta}
\\
&\quad + L^{m-p} |\bar a|^{\upkappa \frac{\mathrm{d}_2-1}{\mathrm{d}_2}+sp-(p-1)} (\log^2(|\abar|))^{\upzeta}.
\end{align*}
Since $|\abar|\to 0$ as $L\to\infty$ by \eqref{AB03}, the above cannot hold for large enough $L$. This contradiction leads to
$\Phi\leq 0$ in $B_2\times B_2$ for some large $L$, proving Lipschitz continuity of $u$ in $B_1$.

\medskip

\noindent{\bf Case 2.} ($sp= p-1$ and  $1<m$).
Since $1<m\leq p-1$ is covered by the proof in Case 1, we assume $m>p-1$.
From Theorem~\ref{Tmain-1}, $u$ is locally $\upkappa-$H\"older for any $\upkappa<1$.  
As is done in Case 1, we take $\varphi=\tilde\varphi(\frac{r_\circ}{3}\cdot)$, the Lipschitz profile function in \eqref{E2.2}. As before, we would like to show that
$\Phi\leq 0$ in $B_2\times B_2$ for some large $L$ dependent on $H, n, s, p$ and $A$. We argue by contradiction, assuming \eqref{E2.3}, and 
arrive at \eqref{E2.7}.

From \eqref{AB03} we have, 
$$L|\bar a| \leq \norm{u}_{C^{0, \tilde\upkappa}(\bar{B}_{\varrho_2})}|\bar a|^{\tilde\upkappa} \Rightarrow L|\bar a|^{1-\tilde\upkappa} \leq C,$$
where $C$ depends on $A$ and $\tilde\upkappa$.
Therefore, $L|\bar a|^\varepsilon \leq C_{\varepsilon}$ for all $\varepsilon \in (0,1)$. 
Fix $\upkappa\in (0, 1)$ and $\varepsilon = \frac{\upkappa(\mathrm{d}_2 -1)}{2(m-1)\mathrm{d}_2}$.
From Lemma~\ref{L3.5}(ii), we get
$$I_3 \geq -C_{\varepsilon_1}|\bar a|^{ \epsilon_\circ},$$
and by Lemmas~\ref{L3.1} and ~\ref{L3.2}, we have
$$ I_1+I_2\geq C L^{p-1}  (\log^2(|\abar|))^{-\upzeta}$$
for all $L\geq L_0$, provided we set $\varepsilon_1$ suitably small. Inserting the estimates in \eqref{E2.7} and using \eqref{Hestim}
we arrive at
\begin{align*}
L^{p-1}  (\log^2|\bar a|)^{-\upzeta} &\leq C \max\{|\abar|^{\upkappa},  |\bar a|^{\upkappa(p-1)}, |\bar a|^{ \epsilon_0}, |\abar|\} +
C L^{m-1} |\bar a|^{\upkappa \frac{\mathrm{d}_2-1}{\mathrm{d}_2}}
\\
& = C \max\{|\abar|^{\upkappa},  |\bar a|^{\upkappa(p-1)}, |\bar a|^{ \epsilon_0}, |\abar|\} +
C L^{m-1} |\bar a|^{2\varepsilon (m-1)}
\\
& \leq C \max\{|\abar|^{\upkappa},  |\bar a|^{\upkappa(p-1)}, |\bar a|^{ \epsilon_0}, |\abar|\} +
C L^{m-1}|\bar a|^{(\frac{1}{2}\wedge\varepsilon) (m-1)} |\bar a|^{\varepsilon (m-1)}
\\
&\leq C \max\{|\abar|^{\upkappa},  |\bar a|^{\upkappa(p-1)}, |\bar a|^{ \epsilon_0}, |\abar|\} +
(C_{\frac{1}{2}\wedge\varepsilon})^{m-1}  |\bar a|^{\varepsilon (m-1)}
\end{align*}
for all $L\geq L_0$. Taking the logarithm-term on the other side leads to
$$ L^{p-1}\leq C \max\{|\abar|^{\upkappa},  |\bar a|^{\upkappa(p-1)}, |\bar a|^{ \epsilon_0}, |\abar|\}(\log^2|\bar a|)^{\upzeta} +
C (C_{\frac{1}{2}\wedge\varepsilon})^{m-1}  |\bar a|^{\varepsilon (m-1)}(\log^2|\bar a|)^{\upzeta},$$
which cannot hold for large enough $L$ (or equivalently, $|\abar|$ small), giving us a contradiction. Now the proof can be completed as in Case 1.

\medskip

\noindent{\bf Case 3}. ($sp>p-1$ and $1<m\leq sp$).
 Since $1<m\leq p-1$ is covered by Case 1, we assume $p-1<m\leq sp$.
The proof in this case is almost the same as in Case 2. By Lemma~\ref{L3.6} we know that $u\in C^{0, \upkappa}_{\rm loc}(B_2)$ for all $\upkappa\in (0, 1)$.
We choose $\varphi$ to be the Lipschitz profile function in \eqref{E2.2}. Fix $\upkappa$ close to 1 such that $sp-(p-1) - (m-1)(1-\upkappa) >0$.
From \eqref{AB03} we also have $L |\abar|^{1-\upkappa}\leq C_\upkappa$ for some $C_\upkappa$ dependent on $A, \varrho_2$ and $\upkappa$.
Now first choosing $\varepsilon_1$ small, and then applying Lemmas~\ref{L3.1}-\ref{L3.5} together with \eqref{Hestim} we obtain from \eqref{E2.7} that
\begin{align*}
L^{p-1}  |\abar|^{(p-1)-sp}(\log^2|\bar a|)^{-\upzeta} &\leq C \max\{|\abar|^{\upkappa},  |\bar a|^{\upkappa(p-1)}, |\abar|\} +
C_{\varepsilon_1}|\abar|^{(p-1)-sp+\epsilon_\circ}
 +C L^{m-1} |\bar a|^{\upkappa \frac{\mathrm{d}_2-1}{\mathrm{d}_2}}
\\
& \leq C \max\{|\abar|^{\upkappa},  |\bar a|^{\upkappa(p-1)}, |\abar|\} +
C_{\varepsilon_1}|\abar|^{(p-1)-sp+\epsilon_\circ} 
\\
&\qquad + C (L|\abar|^{1-\upkappa})^{m-1} |\bar a|^{-(m-1)(1-\upkappa)}
\\
&\leq C \max\{|\abar|^{\upkappa},  |\bar a|^{\upkappa(p-1)}, |\abar|\} +
C_{\varepsilon_1}|\abar|^{(p-1)-sp+\epsilon_\circ} 
 + C C_\upkappa^{m-1}  |\bar a|^{-(m-1)(1-\upkappa)},
\end{align*}
which implies,
\begin{align*}
L^{p-1} & \leq C \max\{|\abar|^{\upkappa},  |\bar a|^{\upkappa(p-1)}, |\abar|\}|\abar|^{sp-(p-1)}(\log^2|\bar a|)^{\upzeta}
\\
&\quad + C_{\varepsilon_1}|\abar|^{\epsilon_\circ}(\log^2|\bar a|)^{\upzeta} + C^{m-1}  |\bar a|^{sp-(p-1)-(m-1)(1-\upkappa)}(\log^2|\bar a|)^{\upzeta}
\end{align*}
for all $L\geq L_0$ and $L_0$ is fixed suitably large from Lemmas ~\ref{L3.1}-\ref{L3.5}. As before, we get a contradiction for large $L$ which completes the proof in this case.

\medskip

\noindent{\bf Case 4.} ($sp>p-1$ and $sp<m$).
From Theorem~\ref{Tmain-1} we know that $u$ is locally $\frac{m-sp}{m-(p-1)}$-H\"older. Fix any $\gamma < \min \{ 1, \upkappa + \frac{1}{2(p-1)} \}$ where $\upkappa\in [\frac{m-sp}{m-(p-1)}, 1)$. Then, for $\varphi=\varphi_\gamma$ in \eqref{E2.2}, we have from \eqref{AB03} that
\begin{align*}
L |\bar a|^{\gamma} \leq C|\bar a|^{\upkappa}
\Rightarrow L |\bar a|^{1-\upkappa} \leq C |\bar a|^{1-\gamma}
\Rightarrow L^{m-(p-1)} |\bar a|^{sp-(p-1)} \leq C |\bar a|^{(1-\gamma)(m-(p-1))},
\end{align*}
using the fact $1-\upkappa\leq 1-\frac{m-sp}{m-(p-1)}=\frac{sp-(p-1)}{m-(p-1)}$.
Now proceeding as before with the help of \eqref{E2.7} and the Lemmas~\ref{L3.1}--\ref{L3.5} and \eqref{Hestim} we arrive at
$$ L^{p-1} |\bar a| ^{\gamma(p-1) -sp} -C [|\abar|^{\gamma(p-1) - sp } + |\bar a|^{\upkappa}+ |\bar a|^{\upkappa(p-1)}+|\abar|] \leq C L^{m-1} 
|\abar|^{(\gamma-1)(m-1)+\upkappa\frac{\mathrm{d}_2-1}{\mathrm{d}_2}}.$$
This implies
\begin{align*}
(1   - \frac{C}{L^{p-1}}) & \leq  \frac{1}{L^{p-1}}|\bar a| ^{sp-\gamma(p-1)} \max\{ |\bar a|^{\upkappa}, |\bar a|^{\upkappa(p-1)}, |\bar a|\}
+ \frac{C}{L} L^{m-(p-1)}|\bar a|^{sp-\gamma (p-1)} |\abar|^{(\gamma-1)(m-1)}
\\
 &\leq  \frac{1}{L^{p-1}}|\bar a| ^{sp-\gamma(p-1)} \max\{ |\bar a|^{\upkappa}, |\bar a|^{\upkappa(p-1)}, |\bar a|\}
+ \frac{C}{L} |\bar a|^{(1-\gamma)(m-(p-1))} |\abar|^{(\gamma-1)((m-1)-(p-1))}
\\
&\leq \frac{1}{L^{p-1}}|\bar a| ^{sp-\gamma(p-1)} \max\{ |\bar a|^{\upkappa}, |\bar a|^{\upkappa(p-1)}, |\bar a|\}
+ \frac{C}{L} |\bar a|^{(1-\gamma)}
\end{align*}
for all large $L$. But this cannot hold for large $L$, giving us $\gamma$-H\"older continuity.
Therefore, we can apply the bootstrapping method as in Lemma~\ref{L3.6} to conclude that $u$ is locally $\upkappa-$H\"older 
continuous for any $\upkappa<1$.

Now, to prove the Lipschitz regularity, we fix $\upkappa<1$ such that $1-\upkappa \leq \frac{sp-(p-1)}{2(m-1)}$. From \eqref{AB03},
with $\varphi$ being the Lipschitz profile, we have $L|\bar a| \leq |\bar a|^{\upkappa} \Rightarrow L|\bar a|^{1-\upkappa} \leq C_\upkappa$. 
As is done in Case 3, \eqref{E2.7} would give us
\begin{align*}
L^{p-1}  |\abar|^{(p-1)-sp}(\log^2|\bar a|)^{-\upzeta} \leq C \max\{|\abar|^{\upkappa},  |\bar a|^{\upkappa(p-1)}, |\abar|\} +
C_{\varepsilon_1}|\abar|^{(p-1)-sp+\epsilon_\circ}
 +C L^{m-1} |\bar a|^{\upkappa \frac{\mathrm{d}_2-1}{\mathrm{d}_2}},
\end{align*}
implying
\begin{align*}
L^{p-1} (\log^2|\bar a|)^{-\upzeta} &\leq C \max\{|\abar|^{\upkappa},  |\bar a|^{\upkappa(p-1)}, |\abar|\} |\abar|^{sp-(p-1)} +
C_{\varepsilon_1}|\abar|^{\epsilon_\circ}
 +C (L|\abar|^{1-\upkappa})^{m-1} |\bar a|^{\frac{sp-(p-1)}{2}}
 \\
&\leq C \max\{|\abar|^{\upkappa},  |\bar a|^{\upkappa(p-1)}, |\abar|\} |\abar|^{sp-(p-1)} +
C_{\varepsilon_1}|\abar|^{\epsilon_\circ}
 + C C_\upkappa^{m-1} |\bar a|^{\frac{sp-(p-1)}{2}},
\end{align*}
which in turn, gives us
\begin{align*}
L^{p-1} 
&\leq C \max\{|\abar|^{\upkappa},  |\bar a|^{\upkappa(p-1)}, |\abar|\} |\abar|^{sp-(p-1)} (\log^2|\bar a|)^{\upzeta}
+ C_{\varepsilon_1}|\abar|^{\epsilon_\circ} (\log^2|\bar a|)^{\upzeta}
\\
&\qquad + C |\bar a|^{\frac{sp-(p-1)}{2}}(\log^2|\bar a|)^{\upzeta}
\end{align*}
for all $L$ large. But this is again a contradiction, giving us local Lipschitz regularity as before.

For the proof of Lemma~\ref{L3.6}, and Cases 3-4, it can be easily seen that the Lipschitz property of 
$f$ is not required for the proof to go through when $sp>p-1$. In fact, merely assuming $f\in C(\bar{B}_2)$ we can deduce the same conclusion. This gives us the last part of Theorem~\ref{Tmain-2}.
\end{proof}


\section{Proof of Theorem~\ref{Tmain-3}}\label{s-Tmain-3}
In this section, we prove the Liouville theorem stated in Theorem~\ref{Tmain-3}. 
First, by Theorems~\ref{Tmain-1}, \ref{Tmain-2} and ~\ref{T1.3}, we note that any solution of
\begin{equation}\label{E4.1}
(-\Delta_p)^s u + H(\grad u) = 0 
\quad \text{in } \mathbb{R}^n
\end{equation}
is locally $\upkappa_0$-H\"older continuous for some $\upkappa_0 \in (0,1]$. 
Moreover, if $u$ is bounded, then $A$ is bounded by
$C\norm{u}_\infty$,  implying $u$ is globally $\upkappa_0$-H\"older continuous, that is,
\[
|u(x) - u(y)| \le C |x-y|^{\upkappa_0}
\quad \text{for all } x,y \in \mathbb{R}^n.
\]

We now outline the strategy of the proof which also uses an Ishii-Lions type argument as in Section~\ref{s-Tmain-2}, but in a slightly different way.
Let $M := \sup_{x\in\Rn}|u(x)|$. For $R>1$, we define 
$$\eta(R) =\frac{1}{R^\beta},$$
 where $\beta \in (0,1) $ to be fixed later.
Let $\psi$ be a smooth cutoff function satisfying $\psi(x) =0$ for $|x| \leq \frac{1}{4}$, $\psi(x) = 2M$ for $|x| \geq \frac{1}{2}$ and $\psi(x) \in [0,2M]$ for $\frac{1}{4} <|x| <\frac{1}{2}$. We choose $0<\gamma < \min \{ 1, \upkappa_0, \frac{sp}{p-1} \}$
 and also let $\beta < \gamma$. Since $u$ is bounded and $\upkappa_0$-H\"older, it is also globally $\gamma$-H\"older. 
 
Define the doubling function:
$$\Phi_R(x,y) =\Phi(x,y):= u(x) - u(y) -\eta(R) |x-y|^{\gamma} -\psi({x}/{R}).$$
Our goal is to show that there exists $R_0 >0$ such that for all $R \geq R_0$, we have
 \begin{equation}\label{E4.2}
\sup_{(x, y)\in \bar{B}_R\times \bar{B}_R} \Phi_R(x,y) \leq 0.
 \end{equation}
Once we have \eqref{E4.2}, the Liouville property follows. More precisely, for any two given points $x, y\in\Rn$, we can find 
$R$ so that $x, y\in B_{\frac{R}{4}}$ for all large $R$. Then \eqref{E4.2} gives us $|u(x)-u(y)|\leq \eta(R) |x-y|^\gamma$. Since
$\eta(R)\to 0$ as $R\to\infty$, letting $R\to\infty$, we obtain $u(x)=u(y)$. Thus, $u$ must be constant.

As done in Section~\ref{s-Tmain-2}, to prove \eqref{E4.2}, we argue by contradiction.  Thus, we start by assuming that for some large $R$
\begin{equation}\label{maxi}
\Phi(\bar{x},\bar{y})=\max_{(x, y)\in \bar{B}_R\times \bar{B}_R} \Phi(x, y)>0,
\end{equation}
where $\bar{x},\bar{y}\in \bar{B}_R$. By the definition of $\psi$ we see that $\Phi\leq 0$ for $|x|\ge R/2$
, giving us $\bar{x}\in B_{R/2}$. Since 
$$\eta(R) {R}^\gamma \rightarrow \infty \quad \text{as} \quad R \rightarrow \infty,
\quad \text{and}\quad \Phi(\bar{x},\bar{y})>0\Rightarrow \eta|\bar{x}-\bar{y}|^\gamma\leq u(\xbar)-u(\ybar),$$  
it follows that
$|\bar{x}-\bar{y}|\leq \frac{R}{4}$ for all large $R$. We set $\bar{a}=\bar{x}-\bar{y}$. Since $\Phi(\bar{x},\bar{y})>0$, it follows that $\bar{a}\neq 0$.
For simplicity, we would write $\eta(R)$ as $\eta$ in the calculations below. We denote 
\begin{align*}
\phi(x, y)&:=\eta |x-y|^\gamma + \psi(x/R)
\\
 \bar{p} &:=\nabla_x\phi(\bar{x},\bar{y}) = \eta \gamma |\bar{a}|^{\gamma-2} \bar{a}+ R^{-1} \nabla\psi(\bar{x}/R), \nonumber
\\
\bar{q}&:=-\nabla_y\phi(\bar{x},\bar{y})=\eta\gamma |\bar{a}|^{\gamma-2}\bar{a}.\nonumber
\end{align*}
Also, for $|z|\leq R/4$, we have
\begin{align*}
\Phi(\bar x, \bar y) \geq \Phi(\bar x + z , \bar y) 
\quad \Rightarrow\quad  u(\bar x + z) \leq \phi(\bar x + z, \bar y) + u(\bar x) - \phi(\bar x, \bar y),
\end{align*}
and 
\begin{align*}
\Phi(\bar x, \bar y) \geq \Phi(\bar x  , \bar y + z)
\quad \Rightarrow\quad  u(\bar y + z) \geq -\phi(\bar x , \bar y + z) + u(\bar y) + \phi(\bar x, \bar y).
\end{align*}
Thus
$$x \mapsto u(x) - \phi(x, \bar y) \text{ has a local maximum point at }\bar x , $$
and
$$y \mapsto u(y) + \phi(\bar x,  y) \text{ has a local minimum point at }\bar y  .$$
For some $\delta<\frac{R}{4}$, to be chosen later, we define the following test functions:
 \[
w_1(z)=\left\{\begin{array}{ll}
\phi(z, \ybar) + \kappa_{\xbar} & \text{if}\; z\in B_\delta(\bar x),
\\[2mm]
u(z) & \text{otherwise},
\end{array}
\right.
\quad \text{and}\quad
w_2(z)=\left\{\begin{array}{ll}
-\phi(\xbar, z)+ \kappa_{\ybar} & \text{if}\; z\in B_\delta(\bar y),
\\[2mm]
u(z) & \text{otherwise},
\end{array}
\right.
\]
with $\kappa_{\xbar}=u(\xbar)- \phi(\xbar, \ybar)$, and $\kappa_{\ybar}= u(\ybar) + \phi(\xbar, \ybar)$.
Applying the definition of viscosity solutions to \eqref{E4.1} we get 
$$(-\Delta_p)^{s}w_1(\bar x) + 
H(\nabla w_1 (\bar x)) \leq 0 \text{\quad and \quad } 
(-\Delta_p)^{s}w_2(\bar y) + H(\nabla w_2 (\bar y)) \geq 0.$$
Subtracting the two viscosity inequalities, we obtain
\begin{equation}\label{viscineq2}
(-\Delta_p)^{s}w_1(\bar x) - (-\Delta_p)^{s}w_2(\bar y) \leq  H(\nabla w_2 (\bar y)) - 
H(\nabla w_1 (\bar x)).
\end{equation}
As was done in Section~\ref{s-Tmain-2},
we consider the following  domains.
$$\cone=\{z\in B_{\delta_0|\abar|}\; :\; |\langle \abar, z\rangle|\geq (1-\eta_0) |\abar||z|\},
\quad \cD_1 = B_{\delta} \cap \cone^c, \quad \text{and}\quad \cD_2=B_{\frac{R}{4}}\setminus (\cD_1\cup\cone),$$
where $\delta_0=\eta_0\in (0, \frac{1}{2})$ will be chosen later,  and,  $\delta = \varepsilon_1 |\abar| \ll \delta_0 |\abar|$.
Also, recall the notation $\mathcal{I}$ from \eqref{E3.6} and write 
\eqref{viscineq2} as
\begin{align}\label{E4.5}
&\underbrace{\cI[\cone] w_1(\bar x)-\cI[\cone] w_2(\bar y)}_{=I_1} + \underbrace{\cI[\cD_1] w_1(\bar x)-\cI[\cD_1] w_2(\bar y)}_{=I_2} + \underbrace{\cI[\cD_2] w_1(\bar x)-\cI[\cD_2] w_2(\bar y)}_{=I_3}  \nonumber
\\ 
&+ \underbrace{\cI[B^c_{\frac{R}{4}}] w_1(\bar x)-\cI[ B^c_{\frac{R}{4}}] w_2(\bar y)}_{=I_4}  \leq  \underbrace{H(\nabla w_2 (\bar y)) - 
H(\nabla w_1 (\bar x))}_{\mathcal{H}}.
\end{align}
Our main goal is to estimate the terms $I_i, i=1,2,3,4$ and $\mathcal{H}$ suitably so that we get a contradiction from \eqref{E4.5}
 for all $R$ large.

We denote $\triangle^2 f(x,z) =f(x) - f(x+z) +\nabla f(x)\cdot z$. Then the following estimate will be useful in our calculations below.
There exists $\delta_0=\eta_0\in (0, \frac{1}{2})$, so that
\begin{align}
\frac{1}{C} \eta(R) |\abar|^{\gamma-2}|z|^2 &\leq \triangle^2\phi (\cdot, \ybar)(\xbar, z), \triangle^2\phi(\xbar, \cdot)(\ybar, z)  \leq C \eta(R) 
|\abar|^{\gamma-2}|z|^2\quad \text{for all} \; z\in \cone, \label{E4.6}
\\
-C \eta(R) |\abar|^{\gamma-2}|z|^2 &\leq \triangle^2\phi (\cdot, \ybar)(\xbar, z), \triangle^2\phi(\xbar, \cdot)(\ybar, z)\leq C \eta(R) 
|\abar|^{\gamma-2}|z|^2
\quad \text{for all}\; |z|\leq \varepsilon_1|\abar|,\label{E4.7}
\end{align}
for all $R\geq R_0$ and $\varepsilon_1\in (0, \frac{1}{2})$, where $R_0$ depends on $\delta_0, \psi$.
A proof of \eqref{E4.6} can be found in \cite[Lemma~2.2]{BT25} and see the proof of \cite[Lemma~3.2]{BT25} for the estimate \eqref{E4.7}.
We also use the fact 
$$|\abar|^{1-\gamma}[\eta(R)R]^{-1}\leq R^{\beta-\gamma}\to 0$$
as $R\to\infty$, for the estimates above. 

\medskip

\noindent\underline{Estimation of $I_1$.} We denote by $\mathfrak{p}(z)=-\grad_x\phi(\xbar,\ybar)\cdot z$. By the anti-symmetry
and linearity of $\mathfrak{p}$
it follows that $\mathcal{I}[D]\mathfrak{p}(x)=0$ for all $x$, provided $D$ is symmetric about $0$.
Also, since $\eta(R)=R^{-\beta}$, we have
$$|\abar|^{1-\gamma}R^{\beta-1}\leq \frac{1}{4^{1-\gamma}} R^{\beta-\gamma}\to 0\quad \text{as}\quad R\to\infty,$$
and therefore,
\begin{equation}\label{E4.8}
\frac{1}{C} \eta |\abar|^{\gamma-1}|z|\leq |\mathfrak{p}(z)|\leq C \eta |\abar|^{\gamma-1}|z|
\end{equation}
for all $R\geq R_0$ and $z\in\cone$, where $R_0$ depends on $\gamma, \psi$.
Now we compute, for $p\geq 2$,
\begin{align*}
\mathcal{I} [\cone] w_1 (\bar x) &= \cI [\cone] w_1 (\bar x) - \cI [\cone] \mathfrak{p} (0)
\\
&\geq \cI [\cone] \phi(\cdot , \bar y)(\bar x)) - \cI [\cone] \mathfrak{p} (0)
\\
&=(p-1)\int_{\cone} \int_0^1 |\mathfrak{p}(z) + t\triangle^2\phi(\cdot , \bar y) (\bar x, z)|^{p-2} \triangle^2\phi(\cdot , \bar y) (\bar x, z) \frac{\dz}{|z|^{n+sp}}
\\
&\geq C\int_{\cone}  |\mathfrak{p}(z)|^{p-2} \eta |\bar a|^{\gamma-2} |z|^2 \frac{\dz}{|z|^{n+sp}}
\\
&\geq C\int_{\cone}  |\eta |\bar a|^{\gamma-1} |z||^{p-2} \eta |\bar a|^{\gamma-2} |z|^2 \frac{\dz}{|z|^{n+sp}}
\\
&= C \eta^{p-1} |\bar a|^{\gamma(p-1) -p}\int_{\cone}   |z|^p \frac{\dz}{|z|^{n+sp}}
\\
&=C \eta^{p-1} |\bar a|^{\gamma(p-1) -p} |\eta_0|^{\frac{n-1}{2}} (\delta_0 |\bar a|)^{p-sp}
\\
&=C \eta^{p-1} |\bar a|^{\gamma(p-1) -sp},
\end{align*}
where in the second line we use the monotonicity property of $J_p$, in the fourth line we use \eqref{E4.6} and inequality from \cite[Lemma~3.3]{KKL19}, in the fifth line we use \eqref{E4.8} and
in the seventh line we compute the integral from \cite[Example 1]{BCCI12} .

When $p \in (1,2)$,  since $p-2<0$, we use the upper bound given by 
\begin{equation}\label{E4.9}
|\mathfrak{p}(z) + t \triangle^2 \phi(\cdot, \bar y)(\bar x, z)| \leq C\eta |\bar a|^{\gamma-1}|z| + C\eta |\bar a|^{\gamma-2} |z|^2 \leq C \eta
 |\bar a|^{\gamma-1}|z|
 \end{equation}
for large $R$ and $|z|\leq \delta_0|\abar|$. Therefore, a similar estimate also holds for $p\in (1, 2)$.

An upper bound for $\cI[\cone] w_2(\bar y)$ is obtained similarly, namely,
$$\cI[\cone]w_2(\bar y) \leq - C \eta^{p-1} |\bar a|^{\gamma(p-1) -sp}$$
for all $R\geq R_0$.
Therefore, 
\begin{equation}\label{I1}
I_1 \geq C \eta^{p-1} |\bar a|^{\gamma(p-1) -sp}
\end{equation}
for all $R\geq R_0$, where $R_0$ depends on $\gamma, \beta, \delta_0$ and $\psi$.

\medskip

\noindent\underline{Estimation of $I_2$.}
First we suppose $p\geq 2$. Then, as $\delta=\varepsilon_1|\abar|$, we get
\begin{align*}
\left|\cI [\cD_1] w_1 (\bar x)\right| &= \left| \cI [\cD_1] w_1 (\bar x) - \cI [\cD_1] \mathfrak{p} (0) \right|
\\
&=\left| (p-1)\int_{\cD_1} \int_0^1 |\mathfrak{p}(z) + t \triangle^2\phi(\cdot , \bar y) (\bar x, z)|^{p-2} \triangle^2\phi(\cdot , \bar y) (\bar x, z) \frac{\dz}{|z|^{n+sp}}\right|
\\
&\leq C \int_{\cD_1}  |\eta |\bar a|^{\gamma-1} |z||^{p-2} \delta |\bar a|^{\gamma-2} |z|^2 \frac{\dz}{|z|^{n+sp}}
\\
&=C \varepsilon_1^{p-sp}\eta^{p-1} |\bar a|^{\gamma(p-1) -sp},
\end{align*}
where in the third line we use \eqref{E4.9}.

Now let $p\in (1,2)$. In view of \eqref{E4.7}, we get
$$  |\triangle^2\phi(\cdot , \bar y) (\bar x, z)|\leq C \eta |\abar|^{\gamma-2}|z|^2:=\tau |z|^2\quad \text{for all}\; z\in B_{\varepsilon_1 |\abar|},$$
for some constant $C$, independent of $R$. Letting
$\xi=\grad_x\phi(\xbar,\ybar)$, we note that
$$ |\xi|\geq \kappa \eta |\abar|^{\gamma-1}$$
for some constant $\kappa$ and all $R$ large.
Following a similar approach as in Lemma \ref{L3.2},
\begin{align*}
\left|\cI [\cD_1] w_1 (\bar x)\right| &= \left| \cI [\cD_1] w_1 (\bar x) - \cI [\cD_1] \mathfrak{p} (0) \right|
\\
&\leq C_p\int_{\cD_1}
\big(|\mathfrak{p}(z)|+|\triangle\phi(\cdot,\ybar)(\xbar,z)|\big)^{p-2}
\big|\triangle\phi(\cdot,\ybar)(\xbar,z)\big|\,\frac{\dz}{|z|^{n+sp}}
\\
&\leq C_p\int_{B_{\varepsilon_1 |\abar|}}
\big(|\mathfrak{p}(z)|+\tau|z|^{2}\big)^{p-2}\,\tau|z|^{2}\,\frac{\dz}{|z|^{n+sp}}
\\
&= C|\xi|^{p-2}\int_{B_{\varepsilon_1 |\abar|}}
\Big(|\langle\hat\xi,z\rangle|+\frac{\tau}{|\xi|}|z|^{2}\Big)^{p-2}
\tau|z|^{2}\,\frac{\dz}{|z|^{n+sp}}
\\
&\leq C(\kappa \eta |\bar a|^{\gamma-1})^{p-2} \tau
\int_{0}^{\varepsilon_1 |\abar|}\!\!\int_{\mathbb{S}^{n-1}}
\Big(|\langle\hat\xi,\theta\rangle|+\frac{\tau r}{|\xi|}\Big)^{p-2}
r^{p}\,\frac{{\rm d}\theta\,\dr}{r^{1+sp}}
\\
&\leq C \eta^{p-1} |\bar a|^{(\gamma-1)(p-1)-1}
\int_{0}^{\varepsilon_1 |\abar|}\Big(1+\frac{\tau r}{|\xi|}\Big)^{p-2}
r^{p}\,\frac{\dr}{r^{1+sp}}
\\
&\leq C \eta^{p-1} |\bar a|^{(\gamma-1)(p-1)-1}
\int_{0}^{\varepsilon_1 |\abar|}r^{p-1-sp}\,\dr
\\
&=C \varepsilon_1^{p-sp}\eta^{p-1} |\bar a|^{\gamma(p-1) -sp},
\end{align*}
where the third line uses the monotonicity of $b\mapsto(a+b)^{p-2}b$ for
$a,b\geq 0$, the fourth the identity $|\mathfrak{p}(z)|=|\xi|\,|\langle\hat\xi,z\rangle|$; $\hat \xi = \frac{\xi}{|\xi|}$
and the sixth the angular estimate of \cite[Lemma~3.5]{KKL19} and the seventh line
uses $p\leq 2$. 
We get a similar estimate for $\cI[\cD_1] w_2(\bar y) $ and combining them we have 
\begin{equation}\label{I2}
I_2 \geq -C \varepsilon_1^{p-sp}\eta^{p-1} |\bar a|^{\gamma(p-1) -sp}
\end{equation}
for all $R\geq R_0$, where $R_0$ is chosen from \eqref{E4.8}. In view of \eqref{I1} and \eqref{I2}, we can set $\varepsilon_1\in (0,\frac{1}{2})$ small enough so that for some $R_0$ we have
\begin{equation}\label{I1+I2}
I_1 + I_2 \geq C \eta^{p-1} |\bar a|^{\gamma(p-1) -sp},
\end{equation}
for all $R\geq R_0$.

\medskip

\noindent\underline{Estimation of $I_3$.} We fix the above choice of $\varepsilon_1$, giving us \eqref{I1+I2}. First we suppose $p>2$.
Denote $\triangle f(x,z) = f(x)-f(x+z)$. For all $|z|\leq R/4$, we have from \eqref{maxi} that
$$
\Phi (\bar x+z, \bar y+z) \leq \Phi(\bar x, \bar y) 
\quad \Rightarrow\quad \triangle u(\bar x, z) - \triangle u (\bar y, z) \geq \triangle \psi\left(\frac{\bar x}{R}, z\right).$$
 Again, $u$ is globally $\gamma$-H\"older continuous. Thus
\begin{align}\label{E4.13}
&\cI[\cD_2] w_1(\bar x)-\cI[\cD_2] w_2(\bar y) \nonumber
\\
&= \int_{\cD_2} \left[ J_p(u(\bar x) -u(\bar x+z)) -J_p(u(\bar y ) -u(\bar y +z)) \right] \frac{\dz}{|z|^{n+sp}}\nonumber
\\
&=\int_{\cD_2} (p-1)\int_0^1 \left| (\triangle u(\bar y, z)) + t \left( \triangle u(\bar x, z) -\triangle u(\bar y, z) \right) \right|^{p-2} \left( \triangle u(\bar x, z) -\triangle u(\bar y, z) \right) \frac{\dz}{|z|^{n+sp}}\nonumber
\\
&\geq -C \int_{\cD_2}  \left| z \right|^{\gamma(p-2)} \left|\triangle \psi\left(\frac{\bar x}{R}, z\right)\right| \frac{\dz}{|z|^{n+sp}}\nonumber
\\
&\geq -C\norm{\grad\psi}_\infty \int_{\cD_2}  \left| z \right|^{\gamma(p-2)} \dfrac{1}{R} |z| \frac{\dz}{|z|^{n+sp}}\nonumber
\\
&= -\dfrac{C}{R} R^{\gamma(p-2) +1 -sp-n} \int_{\varepsilon_1 |\bar a|  \leq |z| \leq \frac{R}{4}} \left[\frac{|z|}{R} \right]^{\gamma(p-2) +1 -sp-n} \dz
\nonumber
\\
&\geq -\dfrac{C}{R} R^{\gamma(p-2) +1 -sp-n} \int_{\varepsilon_1 |\bar a|  \leq |z| \leq \frac{R}{4}} \left[\frac{|z|}{R} \right]^{\gamma(p-1)  -sp-n} \dz
\nonumber
\\
&\geq -\frac{C}{R^{\gamma}} \int_{\varepsilon_1 |\bar a|  \leq |z| \leq \frac{R}{4}} |z| ^{\gamma(p-1) -sp-n} \dz \nonumber
\\
&\geq -\frac{C}{R^\gamma} (\varepsilon_1 |\bar a|)^{\gamma(p-1)-sp}.
\end{align}

Now suppose $p\in (1,2]$ . We use the following algebraic inequality
$$|J_p(a) - J_p(b)| \leq 2|a-b|^{p-1}$$
for $a, b\in\R$. Then,
\begin{align}\label{E4.14}
 \cI[\cD_2] w_1(\bar x)-\cI[\cD_2] w_2(\bar y) 
 &\geq  \int_{\cD_2} \left[ J_p(\triangle u(\bar y, z) + \triangle \psi(\xbar/R, z)) -J_p(\triangle u(\bar y, z )) \right] \frac{\dz}{|z|^{n+sp}}\nonumber
 \\
 &\geq -2\int_{\cD_2} \left|\triangle \psi\left(\frac{\bar x}{R}, z\right)\right|^{p-1} \frac{\dz}{|z|^{n+sp}}\nonumber
 \\
 &\geq -C\norm{\grad\psi}_\infty \int_{\varepsilon_1 |\bar a|  \leq |z| \leq \frac{R}{4}} \dfrac{1}{R^{p-1}} |z|^{p-1} \frac{\dz}{|z|^{n+sp}}\nonumber
  \\
 &\geq -\dfrac{C}{R^{p-1}} R^{(p-1) - n - sp} \int_{\varepsilon_1 |\bar a|  \leq |z| \leq \frac{R}{4}}  \left[\frac{|z|}{R}\right]^{p-1 - n - sp} \dz\nonumber
   \\
 &\geq -C R^{- n - sp} \int_{\varepsilon_1 |\bar a|  \leq |z| \leq \frac{R}{4}} \left[\frac{|z|}{R}\right]^{\gamma(p-1) - n - sp} \dz\nonumber
 \\
  &\geq -\frac{C_{\varepsilon_1}}{R^{\gamma(p-1)}}  |\bar a|^{\gamma(p-1) - sp}\nonumber
 \\ 
  &\geq -C_{\varepsilon_1} \eta^{\frac{\gamma}{\beta}(p-1)} |\bar a|^{\gamma(p-1) - sp},
\end{align}
where $C_{\varepsilon_1}$ depends on $\varepsilon_1,n,s, p, \gamma$ and $\psi$.
Now, using \eqref{E4.13} and \eqref{E4.14}, we choose $\beta$ small enough to ensure that $\gamma- \beta(p-1)>0$ and $\frac{\gamma}{\beta}(p-1) -(p-1)>0$. In view of \eqref{I1+I2} and this choice of $\beta$,
we can find $R_0$ large enough so that
\begin{equation}\label{I123}
I_1 + I_2 +I_3 \geq C \eta^{p-1} |\bar a|^{\gamma(p-1) -sp}
\end{equation}
for all $R\geq R_0$.

\medskip

\noindent\underline{Estimation for $I_4$.}
Now we compute a lower bound for $I_4$ as
\begin{align*}
I_4&=\cI[B_{\frac{R}{4}}^c] w_1(\bar x)-\cI[B_{\frac{R}{4}}^c] w_2(\bar y) 
\\
&= \int_{B_{\frac{R}{4}}^c} \left[ J_p(u(\bar x) -u(\bar x+z)) -J_p((u(\bar y ) -u(\bar y +z)) \right] \frac{\dz}{|z|^{n+sp}}
\\
&\geq -2(2\norm{u}_\infty)^{p-1} \int_{B_{\frac{R}{4}}^c} \frac{\dz}{|z|^{n+sp}} = -\frac{C}{R^{sp}}.
\end{align*}
From our choice, $\gamma < \frac{sp}{p-1} \Leftrightarrow sp-\gamma(p-1) > 0$, and $\eta(R) |\bar a|^{\gamma} \leq 2\norm{u}_\infty \Rightarrow |\bar a| \leq C R^{\frac{\beta}{\gamma}}$. Hence, for any constant $\kappa$
\begin{align*}
\eta^{p-1}|\bar a|^{\gamma(p-1) -sp} \geq \frac{\kappa}{R^{sp}}
\Leftarrow \frac{1}{\kappa} \geq R^{-sp + \beta(p-1)} |\bar a|^{sp-\gamma(p-1)} \Leftarrow
\frac{1}{\kappa} \geq R^{-sp + \beta(p-1)} R^{\beta(\frac{sp-\gamma(p-1)}{\gamma})}.
\end{align*}
Now we further set $\beta$ to be small enough so that exponent of $R$ on the RHS becomes negative, and thus, there exists $R_0>0$ such that 
$$ I_4 \geq -\kappa\eta^{p-1} |\abar|^{\gamma(p-1)-sp}$$
for all $R\geq R_0$, where the constant $\kappa$ can be chosen small and $R_0$ would depend on $\kappa$.
Therefore, from \eqref{I123} and the estimate of $I_4$, we can find $R_0$ such that 
\begin{equation}\label{I1234}
I_1 + I_2 +I_3 + I_4 \geq C \eta^{p-1} |\bar a|^{\gamma(p-1) -sp}
\end{equation}
for all $R\geq R_0$.

\medskip

\noindent\underline{Estimation of $\mathcal{H}$.} 
Finally, we estimate $\mathcal{H}$.  Since 
$$\frac{1}{\eta(R) R}|\abar|^{1-\gamma}\leq \frac{C}{R^{\gamma-\beta}}\to 0 \quad \text{as}\quad R\to \infty,$$
we have $|\grad_x\phi(\xbar,\ybar)|\leq \eta(R) |\abar|^{\gamma-1}$ for all large $R$. Therefore,
\begin{align}
|\mathcal{H}| &=\left|H(-\grad_y \phi (\bar x, \bar y)) - H(\grad_x \phi (\bar x, \bar y))\right|\nonumber
\\
&\leq C \left[ (|\grad_y \phi (\bar x, \bar y)|^{m-1} + |\grad_x \phi (\bar x, \bar y)|^{m-1}+1)   |\grad\psi(x/R)|  \right] \nonumber
\\
&\leq \frac{C}{R}\left| \eta |\bar a|^{\gamma-1} \right|^{m-1} + \frac{C}{R}  \label{estimateofH}
\end{align}
for all $R\geq R_0$, where $R_0$ is a constant.

Now using  \eqref{I1234} and \eqref{estimateofH} in \eqref{E4.5} we obtain
$$\eta^{p-1} |\bar a|^{\gamma(p-1) -sp} \leq \frac{C}{R} \eta^{m-1}|\bar a|^{(m-1)(\gamma-1)} + \frac{C}{R},$$
leading to
\begin{align}\label{final}
1 &\leq \frac{C}{R} \eta^{m-p}|\bar a|^{\gamma(m-p) -(m-sp) +1} + \frac{C}{R} \eta^{1-p} |\abar|^{sp-\gamma(p-1)}\nonumber
\\
&=\frac{C}{R} R^{-\beta(m-p)}|\bar a|^{\gamma(m-p) -(m-sp) +1} + \frac{C}{R} R^{-\beta(1-p)} |\abar|^{sp-\gamma(p-1)}
\end{align}
for all $R\geq R_0$, and $R_0$ is chosen depending on \eqref{I1234} and \eqref{estimateofH}.

Now to conclude the proof we consider two situations.

First, we suppose $\gamma(m-p) -(m-sp) +1 \geq 0$.
Using $|\bar a| \leq C R^{\frac{\beta}{\gamma}}$, which is implied by $\eta(R) |\bar a|^{\gamma} \leq 2\norm{u}_\infty$, we get
\begin{align*}
\frac{C}{R} R^{-\beta(m-p)}|\bar a|^{\gamma(m-p) -(m-sp) +1} &\leq \frac{C_1}{R} R^{-\beta(m-p) +\frac{\beta}{\gamma}(\gamma(m-p) -(m-sp) +1)},
\\
\frac{C}{R} R^{-\beta(1-p)} |\abar|^{sp-\gamma(p-1)} & \leq \frac{C_1}{R} R^{-\beta(1-p)+\frac{\beta}{\gamma}(sp-\gamma(p-1))}. 
\end{align*}
In this case, we can choose $\beta$ small, if required, so that $-\beta(m-p) +\frac{\beta}{\gamma}(\gamma(m-p) -(m-sp) +1) -1 <0$
and $-\beta(1-p)+\frac{\beta}{\gamma}(sp-\gamma(p-1))-1<0$.  Then
\eqref{final} cannot hold  as $R$ enlarges, which is a contradiction. Hence \eqref{E4.2} holds in this case.

Next, we suppose $\gamma(m-p) -(m-sp) +1 < 0$. Since $\gamma < \min \{ 1, \upkappa_0, \frac{sp}{p-1} \}$, we can find
$\upkappa>\gamma$ such that $u$ is globally $\upkappa$-H\"older continuous. Then from $\Phi(\bar x, \bar y) >0$ we get 
\begin{align*}
\eta |\bar a|^{\gamma} \leq u(\bar x) - u(\bar y) \leq C |\bar a|^\upkappa
\Rightarrow |\bar a|^{\gamma-\upkappa} \leq C R^{\beta}.
\end{align*}
Therefore,
$$\frac{C}{R} R^{-\beta(m-p)}|\bar a|^{\gamma(m-p) -(m-sp) +1} \leq \frac{C_1}{R} R^{-\beta(m-p) +\frac{\beta}{\gamma-\upkappa}(\gamma(m-p) -(m-sp) +1)}.$$
In this case, we choose $\beta$ small, depending on $\gamma, \upkappa, m, p, s$, so that 
\begin{align*}
-\beta(m-p) +\frac{\beta}{\gamma-\upkappa}(\gamma(m-p) -(m-sp) +1) -1 &<0,
\\
-\beta(1-p)+\frac{\beta}{\gamma}(sp-\gamma(p-1))-1&<0.
\end{align*}
 As before,  we again get a contradiction to \eqref{final} as $R \rightarrow \infty$.
Therefore \eqref{E4.2} holds, implying $u$ is a constant. This completes the proof.
\qed

\bigskip
\subsection*{Acknowledgement}
The authors sincerely thank the referees for their careful reading and for
several pertinent comments, which have led to a substantially improved version
of the paper.
Part of this project was done during a visit of A.B.\ to the Instituto de
Matem\'atica of Universidade Federal do Rio de Janeiro. The kind hospitality of the department
is acknowledged.
A.B.\ was partially supported by 
an ANRF-ARG grant ANRF/ARG/2025/000019/MS. 
E.T. was partially supported by CNPq Grant 306022/2023-0, FAPERJ APQ1 Grant 210.573/2024 and FAPERJ APQ2 204.215/2025. 
Both A.B.\ and E.T.\ were also supported by a CNPq Grant 408169. 

\subsection*{Data Availability}  All data generated or analyzed during this study are included in this article.
\medskip

\noindent{\bf Declarations}
\subsection*{Conflicts of Interest}  The authors declare that they have no conflict of interest to this work.

\begin{thebibliography}{77}
\bibitem{ABC19} A. Arapostathis, A. Biswas, and L. Caffarelli. On uniqueness of solutions to viscous HJB equations with
a subquadratic nonlinearity in the gradient. Comm. Partial Differential Equations, 44(12):1466--1480, 2019.

\bibitem{Bar91} G. Barles. A weak Bernstein method for fully nonlinear elliptic equations,
\emph{Diff. and Integral Equations} 4(2),
241--262, 1991

\bibitem{BCI08}
G. Barles, E. Chasseigne, and C. Imbert. On the Dirichlet Problem for Second Order
Elliptic Integro-Diﬀerential Equations. Indiana Univ. Math. J. 57 (2008), no. 1, 213-246.

\bibitem{BCI11} G. Barles, E. Chasseigne, and C. Imbert.
H\"{o}lder continuity of solutions of second-order non-linear elliptic integro-differential equations,
\emph{J. Eur. Math. Soc.} (JEMS) 13,  1--26, 2011

\bibitem{BCCI12} G. Barles, E. Chasseigne, A. Ciomaga, and C. Imbert.
 Lipschitz regularity of solutions for mixed integro-differential equations,
 \emph{ J. Differential Equations}, 252(11), 6012--6060, 2012

\bibitem{BKLT} Guy Barles, Shigeaki Koike, Olivier Ley, and Erwin Topp. Regularity results and large time behavior for integro-
differential equations with coercive Hamiltonians. Calc. Var. Partial Differential Equations, 54(1):539--572, 2015.

\bibitem{BLT17} Guy Barles, Olivier Ley, and Erwin Topp. Lipschitz regularity for integro-differential equations with coercive
Hamiltonians and application to large time behavior. Nonlinearity, 30(2):703--734, 2017.

\bibitem{BM16} Guy Barles and Joao Meireles. On unbounded solutions of ergodic problems in $\mathbb{R}^m$ for viscous Hamilton-Jacobi
equations. Comm. Partial Differential Equations, 41(12):1985--2003, 2016

\bibitem{BT16} Guy Barles and Erwin Topp. Lipschitz regularity for censored subdiffusive integro-differential equations with
superfractional gradient terms. Nonlinear Anal., 131:3--31, 2016.

\bibitem{BF92} A. Bensoussan and J. Frehse. On Bellman equations of ergodic control in $\mathbb{R}^n$. J. Reine Angew. Math., 429:125--160, 1992.

\bibitem{Bern06} S. Bernstein. Sur la g\'en\'eralisation du probl\`eme de Dirichlet. Math. Ann., 62(2):253--271, 1906.

\bibitem{Bern10} S. Bernstein. Sur la g\'en\'eralisation du probl\`eme de Dirichlet. Math. Ann., 69(1):82--136, 1910.


\bibitem{BBS26} M. Bhakta, A. Biswas and A. Sen.
Liouville results for supersolutions of fractional $p$-Laplacian equations with gradient nonlinearities,
to appear in {\it Proc. AMS}, 2026

\bibitem{BHV14} M.F. Bidaut-V\'eron, M. Garc\'{i}a-Huidobro, and L. V\'eron. Local and global properties of solutions of quasilinear
Hamilton-Jacobi equations, J. Funct. Anal. 267: 3294-3331, 2014

\bibitem{BHV19} M.F. Bidaut-V\'eron, M. Garc\'{i}a-Huidobro, and L. V\'eron. Estimates of solutions of elliptic equations with a
source reaction term involving the product of the function and its gradient, Duke Math. J. 168: 1487--1537, 2019

\bibitem{BHV20} M.F. Bidaut-V\'eron, M. Garc\'{i}a-Huidobro, L. V\'eron.
A priori estimates for elliptic equations with reaction terms involving the function and its gradient. \emph{Math. Ann.} 378, 13--56, 2020


\bibitem{BQT25} A. Biswas, A. Quaas, and E. Topp,
Nonlocal Liouville theorems with gradient nonlinearity,
\emph{J. Funct. Anal.}, Volume 289, Issue 8, 15 October 2025.


\bibitem{BS25} A. Biswas and A. Sen.
 Improved H\"{o}lder regularity of fractional $(p, q)$-Poisson equation with regular data,
2025. arXiv: 2507.09920
 
 



 \bibitem{BT24} A. Biswas and E. Topp.
Nonlocal ergodic control problem in $\mathbb{R}^d$, 
\emph{Math. Annalen} 390, 45--94, 2024

\bibitem{BT25} A. Biswas and E. Topp.
Lipschitz regularity of fractional $p$-Laplacian, 
\emph{Annals of PDE} 11, no. 27,  2025

 
 \bibitem{BDLM25} V. B\"ogelein, F. Duzaar, N. Liao, and K. Moring. 
Gradient estimates for the fractional $p$-Poisson equation,
\emph{J.  Math. Pures et Appl.} 204, appeared online, 2025
arXiv:2503.05903, 2025

\bibitem{BDLMS24a} V. B\"ogelein, F. Duzaar, N. Liao, G. Molica Bisci, and R. Servadei. 
Regularity for the fractional $p$-Laplace equation, 
\emph{J. Funct. Anal.} 289(9), 2025

\bibitem{BDLMS24b} V. B\"ogelein, F. Duzaar, N. Liao, G. Molica Bisci, and R. Servadei. 
Gradient regularity of $(s, p)$-harmonic functions, \emph{Calc. Var.} 64, 253, 2025



\bibitem{BL17} L. Brasco and E. Lindgren.
Higher Sobolev regularity for the fractional $p$-Laplace equation in the superquadratic case, 
\emph{Adv. Math.} 304, 300--354, 2017

\bibitem{BLS18} L. Brasco, E. Lindgren, and A. Schikorra.
Higher H\"older regularity for the fractional p-Laplacian in the superquadratic case,
\emph{ Adv. Math.} 338, 782--846, 2018





 
 \bibitem{CDLP} I. Capuzzo Dolcetta, F. Leoni, and A. Porretta. 
 H\"{o}lder estimates for degenerate elliptic equations with coercive
Hamiltonians, \emph{Trans. Am. Math. Soc.}, 362(9):4511--4536, 2010



\bibitem{CD90} V. Chiad\`{o} Piat and  A. Defranceschi.
Homogenization of quasi-linear equations with natural growth terms,
\emph{Manuscr. Math.} 68, 229-247, 1990

\bibitem{CC15} J. Ching and F.C. C\^irstea.
Existence and classification of singular solutions to nonlinear elliptic equations with a gradient term,
\emph{Anal. PDE} 8 , pp. 1931--1962, 2015

\bibitem{CC20} J. Ching and F.C. C\^irstea.
Gradient estimates for nonlinear elliptic equations with a gradient-dependent nonlinearity,
\emph{Proc. R. Soc. Edinb., Sect. A} 150 (3), 1361--1376, 2020

\bibitem{CGT22} A. Ciomaga, D. Ghilli, and E. Topp.
 Periodic homogenization for weakly elliptic Hamilton-Jacobi-Bellman
equations with critical fractional diffusion,
\emph{Comm. Partial Differential Equations}, 47(1):1--38, 2022

\bibitem{CG23} M. Cirant and A. Goffi. On the Liouville property for fully nonlinear equations with superlinear first-order
terms, ``Proceedings of the Conference on Geometric and Functional Inequalities and Recent Topics in Nonlinear
PDEs", Contemporary Mathematics, American Mathematical Society, 781 (2023)

\bibitem{CV22} M. Cirant and G. Verzini.
Local H\"older and maximal regularity of solutions of elliptic equations with superquadratic gradient terms,
\emph{Advances in Mathematics}, Vol. 409, Part B, 2022, 108700

%

\bibitem{Coz17} M. Cozzi.
Regularity results and Harnack inequalities for minimizers and solutions of nonlocal problems: a unified approach via fractional De Giorgi classes,
\emph{ J. Funct. Anal.} 272 , 4762--4837, 2017

\bibitem{DGS09} A. Dall'Aglio, D. Giachetti, and S. Segura de Le\'{o}n.
Global existence for parabolic problems involving the $p$-Laplacian and a critical gradient term,
\emph{Indiana Univ. Math. J.} 58 (1), 1--48, 2009


\bibitem{DQT17} G. Dávila, A. Quaas, and E. Topp, Continuous viscosity solutions for nonlocal Dirichlet problems with coercive gradient terms. Math. Ann. 369 (3-4), 1211-1236, 2017.
 
\bibitem{DPQ25} L.~M. Del~Pezzo and A. Quaas, The fundamental solution of the fractional $p-$Laplacian, \emph{NoDEA Nonlinear Differential Equations Appl.}  33, no.~3, Paper No. 62, 2026


\bibitem{DKP14} A. Di Castro, T. Kuusi, and G. Palatucci.
Nonlocal Harnack inequalities, 
\emph{J. Funct. Anal.} 267 (6), 1807--1836, 2014

\bibitem{DKP16} A. Di Castro, T. Kuusi, and G. Palatucci.
Local behavior of fractional $p$-minimizers, 
\emph{Ann. Inst. H. Poincar\'e Anal. Non Lin\'eaire} 33 , 1279--1299, 2016

\bibitem{DKLN} L. Diening, K. Kim, H.-S. Lee and S. Nowak.
 Higher differentiability for the fractional $p$-Laplacian, 
 \emph{Math. Annalen} 391, 5631--5693, 2025 
 

 
\bibitem{DN25} L. Diening and S. Nowak. 
Calder\'on-Zygmund estimates for the fractional $p$-Laplacian,
\emph{Annals of PDE} 11, no. 6,  2025


\bibitem{F09} R. Filippucci. Nonexistence of positive weak solutions of elliptic inequalities, Nonlinear Anal., 70
2903--2916, 2009


\bibitem{GL24} P. Garain and E. Lindgren.
Higher H\"older regularity for the fractional $p$-Laplace equation in the subquadratic
case. \emph{Math. Annalen} 390, 5753--5792, 2024


\bibitem{GJS25} D. Giovagnoli, D. Jesus and L. Silvestre.
$C^{1+\alpha}$ regularity for fractional $p$-harmonic functions.
ArXiv:2509.26565, 2025

\bibitem{GJS26} D. Giovagnoli, D. Jesus and L. Silvestre.
H\"older gradient estimates for fractional p-caloric functions.
ArXiv:2609.20097, 2026




\bibitem{IMS20} A. Iannizzotto, S. Mosconi, and M. Squassina.
Fine boundary regularity for the degenerate fractional p-Laplacian,
\emph{ J. Funct. Anal.} 279, no. 8, 108659, 54 pp., 2020

\bibitem{IL90} H. Ishii and P. L. Lions.
Viscosity solutions of fully non-linear second-order elliptic partial differential equations,
\emph{J. Differential Equations} 83, No.1, 26--78, 1990

\bibitem{IN}
H. Ishii, and G. Nakamura. A Class of Integral Equations and Approximation of p-Laplace Equations. Calc. Var. PDE (2010), no. 37, 485--522. 


\bibitem{JSU} D. Jesus, A. Sobral, and J. M. Urbano.
Fractional $p$-caloric functions are Lipschitz, ArXiv.2603.12065, 2026




\bibitem{KKL19} J. Korvenp\"a\"a, T. Kuusi and E Lindgren. Equivalence of solutions to fractional $p$-Laplace type equations, 
\emph{J.  Math. Pures  Appl.} 132, 1--26, 2019

\bibitem{KKP16} J. Korvenp\"a\"a, T. Kuusi,  and G. Palatucci.
 H\"{o}lder continuity up to the boundary for a class of fractional obstacle problems,
 \emph{ Atti Accad. Naz. Lincei Rend. Lincei Mat. Appl.} 27, no. 3, 355--367, 2016

\bibitem{LU68} O. Ladyzhenskaya and N. Ural'tseva, Linear and quasilinear elliptic equations. Academic Press, 1968

\bibitem{Lib91} G. M. Lieberman, The natural generalization of the natural conditions of Ladyzhenskaya and Ural'tseva for
elliptic equations, \emph{Comm. Partial Differential Equations} 16, no. 2-3, 311--361, 1991

\bibitem{Lions} P. L. Lions, Quelques remarques sur les probl\`emes elliptiques quasilin\'eaires du
second ordre, \emph{J. Anal. Math.} 45: 234--254,1985.

\bibitem{MP99} \`{E}. Mitidieri and S. I. Pohozaev. Absence of positive solutions for quasilinear elliptic problems in $\mathbb{R}^N$, Tr. Mat.
Inst. Steklova, 227:192--222, 1999

\bibitem{MP01} \`{E}. Mitidieri and S. I. Pohozaev. A priori estimates and the absence of solutions of nonlinear partial differential
equations and inequalities, Tr. Mat. Inst. Steklova, 234:1--384, 2001

\bibitem{Pan} A. Pankov.
G-convergence and Homogenization of Nonlinear Partial Differential Operators, Kluwer Academic Publishers, Dordrecht, 1997

\end{thebibliography}

\end{document}